\documentclass[preprint,11pt]{elsarticleOnline}

\usepackage{amssymb}
\usepackage{amsthm}

\usepackage[margin=1in]{geometry} 
\usepackage[short]{ optidef }
\usepackage{multirow}
\usepackage[colorlinks=true,breaklinks=true,bookmarks=true,urlcolor=blue,
     citecolor=blue,linkcolor=blue,bookmarksopen=false,draft=false]{hyperref}
\usepackage[dvipsnames,table,xcdraw]{xcolor}
\usepackage{siunitx}
\usepackage{multirow}
\usepackage{hypcap}
\usepackage{float}
\usepackage{tikz}
\usetikzlibrary{arrows,positioning,decorations.pathreplacing,decorations.markings}
\usepackage{pgfplots}
\pgfplotsset{compat=1.18}
\usepackage[ruled,vlined,linesnumbered]{algorithm2e}

\usepackage{xr}
\usepackage[utf8]{inputenc}   
\usepackage[T1]{fontenc}      

\newtheorem{proposition}{Proposition}

\newtheorem{observation}{Observation}
\newtheorem{example}{Example}

\newcommand{\Z}{\ensuremath{\mathbb{Z}} }

\newcommand{\rev}[1]{{\color{black} #1}}
\usepackage[graphicx]{realboxes}

\usepackage[symbol]{footmisc}
\usepackage{soul}

\journal{ }

\begin{document}

\begin{frontmatter}



\title{ Strengthening Dual Bounds for Multicommodity Capacitated Network Design with Unsplittable Flow Constraints
}
%


\author[inst2]{Lacy M. Greening\corref{cor1}}

\cortext[cor1]{Corresponding author: lacy.greening@asu.edu}

\affiliation[inst2]{organization={School of Computing and Augmented Intelligence},
            addressline={Arizona State University}, 
            city={Tempe},
            state={Arizona},
            country={United States}}  
            
\affiliation[inst1]{organization={H. Milton Stewart School of Industrial and Systems Engineering},
            addressline={Georgia Institute of Technology}, 
            city={Atlanta},
            state={Georgia},
            country={United States}}

\author[inst1]{Santanu S. Dey}
\author[inst1]{Alan L. Erera}


\begin{abstract}

Multicommodity capacitated network design (MCND) models can be used to optimize the consolidation of shipments within e-commerce fulfillment networks. In practice, fulfillment networks require that shipments with the same origin and destination follow the same transfer path.
This unsplittable flow requirement 
complicates the MCND problem, requiring integer programming (IP) formulations \rev{in which} binary variables replac\rev{e} continuous flow variables. 
To enhance the solvability 
of this variant of the MCND problem for large-scale logistics networks, this work focuses on strengthening dual bounds. We investigate the polyhedra of arc-set relaxations, and we introduce two new classes of valid inequalities that can be implemented within solution approaches.
We develop one approach that dynamically adds valid inequalities to the root node of a reformulation of the MCND IP with additional valid
metric inequalities.
We \rev{show}
the effectiveness of our ideas with a comprehensive computational study using path-based fulfillment instances, constructed from data provided by a large U.S.-based e-commerce company, and the well-known arc-based Canad instances.
Experiments show that our best solution approach for a practical path-based model reduces the IP gap by an average of 26.5\% and 22.5\% for the two largest instance groups, compared to solving the reformulation alone, demonstrating its effectiveness in improving the dual bound. In addition, experiments using only the arc-based relaxation highlight the strength of our new valid inequalities relative to the linear programming relaxation (LPR), yielding an IP-gap reduction of more than 85\%. 



\end{abstract}



\begin{keyword}
Network design, fulfillment logistics, valid inequalities, dual bounds
\end{keyword}

\end{frontmatter}


\section{Introduction} \label{sec:intro} 
The rapid growth of e-commerce has led companies to seek improved fulfillment logistics systems to reduce the costs of shipping customer orders and meeting on-time delivery promises. Optimization approaches for the design of such networks have subsequently received attention \citep{wang2019hybrid,fontaine2021scheduled,CrainicGendreau2021,hewitt2022flexible,greening2023lead,benidis2023middle,satici2024tactical}. A core optimization modeling approach used in this application is the \textit{multicommodity capacitated network design} (MCND) problem, which seeks to optimize the transportation of shipments through a minimum-cost transshipment network \citep{magnanti1984network,minoux1989networks,atamturk2002splittable,frangioni20090,AlperChND}. Shipments are grouped into \textit{commodities}, defined as origin-destination pairs with specific demand volumes, and the design model prescribes one or more paths for the flow of each commodity. The selected paths and assigned shipment flows then require integer multiples of a capacity unit (e.g., truck trailers) to be installed on directed arcs, incurring a fixed cost for each installed unit and possibly a variable cost based on demand volume transported. Costs are shared among commodities that use the same arc, and this modeling approach identifies the best ways to consolidate shipments and improve the fill rate of dispatched trailers. Recent studies show that these approaches can reduce fulfillment network costs by over 30\%. However, achieving these results for the practical networks operated by industry relies on solving large-scale mixed-integer programming models involving thousands of binary and integer variables, as well as tens of thousands of constraints \citep{greeningintegrating}.

In real-world fulfillment networks, there are additional dispatch and network design constraints specific to each organization beyond those typically found in the standard MCND model. A key distinction between practical applications and the standard MCND model is that in practice all shipments of the same commodity must follow the same path of arcs from origin to destination within the network. 
That is, MCND models for fulfillment networks have unsplittable (or non-bifurcated) flow constraints  \citep{atamturk2002splittable,HewittMike2010CEaH,greening2023lead,hewitt2013branch,chen2021exact}. This requirement simplifies the logistics process and reduces errors, since outbound loading destinations remain consistent for shipments of the same commodity. 
However, it can make realistically-sized instances even more challenging to solve to optimality because it requires replacing continuous flow variables with binaries. 
Numerous works focus on developing effective metaheuristic approaches to identify strong primal solutions for these so-called single-path MCND problems \cite{CrainicGendreau2021, Bakir2021-MCSND}; however, assessing the quality of solutions is typically difficult for realistically-sized instances due to weak dual bounds, a common issue in network design problems caused by their weak linear programming (LP) relaxations. In this paper, we aim to enhance the solvability of this specific integer programming (IP) variant of the MCND problem \rev{encountered in transportation and logistics by strengthening its dual bounds.}
Specifically, we:\vspace{-2mm}
\begin{itemize}\setlength\itemsep{-1mm}
    \item reformulate the MCND model to use a multiple-choice binary variable scheme to select arc capacities, as opposed to integer capacity variables, and verify this reformulation enables solvers to find stronger lower bounds;
    \item strengthen the LP relaxation of this reformulation by introducing two new classes of valid inequalities for the unsplittable MCND problem (applicable to both arc- and path-based formulations);
    \item \rev{design an integrated procedure for generating and selecting metric inequalities, alongside additional valid inequalities, that yields stronger dual bounds for the MCND problem;}and
    \item conduct a computational study on two sets of instances---one for path-based formulations and the other for arc-based formulations---to 
    \rev{show the strengthening effect of our valid inequalities on LP relaxations and the performance benefits of our integrated approach.}
\end{itemize}

The remainder of this paper is organized as follows. \rev{In Section \ref{sec:lit}, we review relevant literature.} In Section \ref{sec:form}, we define and model the MCND problem using two path-based formulations, one which uses integer variables to model arc capacities (INT) and another which uses a multiple-choice selection of binary variables to model arc capacities (BIN).
In Section \ref{sec:relax}, we study a structured relaxation of the BIN formulation and introduce two new classes of valid inequalities. In Section \ref{sec:helper}, we describe a set of metric inequalities
\rev{and a procedure for generating and selecting them that, when implemented, improve the dual bounds and overall solution performance of commercial solvers on the MCND instances.}
Finally, in Section \ref{sec:comp}, we report the results of computational experiments performed on instances constructed from e-commerce fulfillment network and demand data, as well as on a set of MCND instances for arc-based formulations.

\section{\rev{Literature Review}}\label{sec:lit}
\rev{The multicommodity capacitated network design (MCND) problem is a central model in network optimization, with applications in telecommunication, transportation, and logistics systems \citep{minoux1989networks,balakrishnan1991models,gavish1991topological,wong2021telecommunications,crainic2000service,wieberneit2008service,Crainic_Gendreau_Gendron_2021}. The problem extends the classical network flow formulation by including fixed-charge or capacity-expansion costs, leading to mixed-integer formulations that jointly determine which arcs to install and how to route multiple commodities \citep{ahuja1988network,magnanti1984network}. 
MCND formulations have been applied to many problems ranging from capacity installation and survivability planning in data networks \citep{minoux1989networks,gavish1991topological,gavish1990backbone,wong2021telecommunications} to service network design in motor carrier, railroad, liner shipping, and public transportation systems \citep{Bakir2021-MCSND,CRAINIC1986,Powell1986,PowellSheffi1983,reinhardt2012branch,laporte2020design}.
For a detailed classification of MCND problem variants, applications, and solution approaches, the reader is referred to \citet{ouorou2000survey}, \citet{salimifard2022multicommodity}, and \citet{Crainic_Gendreau_Gendron_2021}.}

\rev{In this study, we focus on a transportation and logistics variant of the MCND problem in which all commodity flows are unsplittable; that is, each commodity must be routed along a single path represented by binary flow variables \citep{atamturk2002splittable,HewittMike2010CEaH,hewitt2013branch,chen2021exact} in contrast to splittable formulations that allow continuous flow allocations to multiple paths \citep{crainic2001bundle,ghamlouche2003cycle,chouman2017commodity}. The formulation considered here assumes a single facility type that can be installed in integer multiples on each arc with no pre-existing capacity, a useful model for load-planning decisions such as scheduling trailer dispatches between terminals in large-scale logistics networks \citep{greening2023lead,Bakir2021-MCSND,erera2013improved}. 
Many studies have developed specialized heuristic and exact hybrid methods to address the combinatorial structure of unsplittable network design models \citep{hewitt2013branch,erera2013improved,greeningintegrating,eom2025recursive}. 
While these approaches are effective in producing high-quality primal solutions, assessing optimality for large-scale instances remains challenging, as the LP relaxation of the MCND is generally weak and becomes even weaker under the unsplittable flow constraint.}

\rev{Efforts to improve dual bounds for the MCND problem have therefore focused on the arc-set relaxation, with much of the research analyzing its polyhedral structure and developing valid inequalities and separation algorithms}
\citep{atamturk2002splittable,chen2021exact,brockmuller1996designing,van2002polyhedral,benhamiche2016unsplittable}.
\citeauthor{brockmuller1996designing} (\citeyear{brockmuller1996designing}) introduce the $c$-strong inequalities and characterize their necessary and sufficient conditions to be facet-defining. 
\citeauthor{atamturk2002splittable} (\citeyear{atamturk2002splittable}) extend this work by showing that the separation problem of $c$-strong inequalities is $\mathcal{NP}$-hard and that it is sufficient to consider only the subspace of fractional variables. The authors also introduce the \rev{$r$}-split $c$-strong and lifted knapsack cover inequalities, both of which include $c$-strong inequalities as a special case with \rev{$r$}-split $c$-strong inequalities often resulting in greater improvements. \rev{\citeauthor{benhamiche2016unsplittable} (\citeyear{benhamiche2016unsplittable}) addressed the Unsplittable Non-Additive Capacitated Network Design (UNACND) problem and analyzed its arc-set polyhedron using a general class of unitary-step set function polyhedra, deriving the Min-Set I and II inequalities as generalizations of the classic $c$-strong inequalities. The Min-Set inequalities strengthen the arc-set relaxation through its combinatorial structure but under non-additive capacity assumptions, where multiple facilities on the same arc cannot be combined.
}

More recently,  \rev{and most closely related our work,}
\citeauthor{chen2021exact} (\citeyear{chen2021exact}) developed a generalized exact separation algorithm for the unsplittable flow arc-set polyhedron \rev{that accommodates multiple facility types with different capacity levels and allows existing (including negative) capacity. Because their model allows additive capacity across facility types and incorporates existing capacity, their facet-defining inequalities include a constant term.
In contrast, our work focuses on a realistic e-commerce network design setting with a single facility type and no pre-existing capacities, where a binary capacity formulation yields valid inequalities derived from the arc-set polyhedron without a constant term.}

Motivated by these studies, 
\rev{our work proposes} two new classes of valid inequalities \rev{derived from cover-type logic \citep{balas1978facets,gu1999lifted,wolsey1975faces}} applicable to both arc- and path-based unsplittable MCND formulations,
in addition to identifying 
an implementation of metric inequalities \citep{AlperChND,bienstock1998minimum,costa2009benders} \rev{that have been used within decomposition and relaxation frameworks to improve computational performance in MCND.}

\section{Formulations} \label{sec:form}
Let $G=(\mathcal{N},\mathcal{A})$ define a transshipment network, where $\mathcal{N}$ represents a set of locations (nodes) in the network and $\mathcal{A}$ represents a set of directed arcs connecting pairs of locations. Shipment demand is modeled using a set $\mathcal{K}$ of commodities. Each commodity $k\in\mathcal{K}$ has an origin $o(k)\in \mathcal{N}$, destination $d(k)\in \mathcal{N}$, and demand volume $d_k$ that must be shipped from origin to destination along a single path (or sequence of adjoined arcs) through the network. As the number of paths between pairs of locations can be exponential, we define a set of paths $\mathcal{P}_k$ that each commodity $k\in \mathcal{K}$ can use. 
Let $f_a$ be the fixed cost of installing one unit of capacity on arc $a\in \mathcal{A}$, $c_p$ be the variable cost per unit of demand when using path $p\in \cup_{k\in \mathcal{K}}\mathcal{P}_k$, and $q_a$ denote the unit capacity of arc $a\in\mathcal{A}$. 
The MCND problem is to determine a minimum-cost allocation of transportation capacity on network arcs to ensure that commodity demand requirements are satisfied.

\subsection{Integer Capacity Variables}
Let binary variables $x_p$ indicate if commodity $k\in\mathcal{K}$ uses path $p\in \mathcal{P}_k$ or not. Let integer variables $\tau_a$ represent the number of capacity units installed on arc $a\in \mathcal{A}$. The path-based unsplittable MCND problem with integer capacity variables (referred to as INT) is formulated as follows:

{ \small \begin{mini!}[2]
	{\scriptstyle x,\tau}{\sum_{k\in\mathcal{K}}\sum_{p \in \mathcal{P}_k}c_pd_kx_p +
	\sum_{a\in \mathcal{A}}f_a\tau_a}{\label{rblpSM}}{\label{ObjSM}}
	\addConstraint{\sum_{p \in \mathcal{P}_k}x_p}{=1,}{\  \forall \,  k \in \mathcal{K} \label{1stSM}}
	\addConstraint{\sum_{k\in \mathcal{K}}\sum_{\{p\in \mathcal{P}_k|p \ni a\}}d_kx_p}{\leq q_a\tau_{a},\quad}{  \ \forall \,  a\in \mathcal{A} \label{8thSM}}
	\addConstraint{x_p}{\in \{0,1\},}{\ \forall \, p \in \mathcal{P}_k,\, \forall \, k \in \mathcal{K}} \label{xvar}
	\addConstraint{\tau_{a}}{\in \mathbb{Z}_{\geq 0},}{\ \forall \, a \in \mathcal{A}.\label{fvar}} 
 \end{mini!} }

The objective is to minimize the sum of fixed and variable costs. Constraints \eqref{1stSM} ensure that one path is selected for each commodity. Constraints \eqref{8thSM} ensure that the required number of capacity units are installed to transport the flow volume assigned to each arc as defined by selected paths.

\subsubsection{Binary Capacity Variables}\label{sect:binReform}
Previous studies \citep{frangioni20090,croxton2003models,croxton2007variable,bonami2014cut} have shown that by reformulating flow and/or capacity variables to use a \textit{multiple-choice} selection scheme can yield stronger linking constraints and other valid inequalities, producing higher-quality LP relaxations. In this scheme, an expanded set of binary selector variables is introduced to replace either an integer capacity and/or a binary flow variable.
We also elect to use a multiple-choice reformulation of INT \eqref{rblpSM}, and later use this reformulation to define a new class of valid inequalities in Section \ref{sec:relax}.

In our multiple-choice reformulation, we replace integer capacity variables $\tau_a$ for arcs $a\in\mathcal{A}$ with a set of binary variables $y_{at}$ that select an integer capacity $t$ from a set $\mathcal{T}_a$ containing allowable non-zero capacity values. We define $\mathcal{T}_a=\{1,\dots,T^{\max}_a\}$ for each arc $a\in\mathcal{A}$, where $T^{\max}_a$ is the maximum number of capacity units that may be necessary on arc $a\in\mathcal{A}$ and is defined as $T^{\max}_a = \left\lceil\frac{\sum_{k\in \mathcal{K}}\delta^k_ad_k}{q_a}\right\rceil$ with $\delta_a^k=1$ if the path set $\mathcal{P}_k$ for commodity $k\in\mathcal{K}$ contains a path that includes arc $a\in \mathcal{A}$ (i.e., $|\{p\in \mathcal{P}_k|p \ni a\}|\geq 1$)
and $\delta_a^k=0$ otherwise.
Thus, we reformulate the integer capacity variables by making the following replacement:
\begin{equation}\label{eqn:replace}
    \tau_a = \sum_{t\in\mathcal{T}_a}ty_{at}, \quad \forall \, a \in\mathcal{A}.
\end{equation}
The path-based unsplittable MCND problem with multiple-choice binary capacity variables (referred to as BIN) is formulated as follows:
{\small \begin{mini!}[2] 
	{\scriptstyle x,y}{\sum_{k\in\mathcal{K}}\sum_{p \in \mathcal{P}_k}c_pd_kx_p +
	\sum_{a\in \mathcal{A}}f_a\sum_{t\in\mathcal{T}_a}ty_{at}}{\label{rblpB1}}{\label{ObjB1}}
	\addConstraint{\sum_{p \in \mathcal{P}_k}x_p}{=1,}{\  \forall \,  k \in \mathcal{K} \label{1stB1}}
	\addConstraint{\sum_{k\in \mathcal{K}}\sum_{\{p\in \mathcal{P}_k|p \ni a\}}d_kx_p}{\leq q_a\sum_{t\in\mathcal{T}_a}ty_{at},\quad}{  \ \forall \,  a\in \mathcal{A} \label{8thB1}}
	\addConstraint{\sum_{t\in\mathcal{T}_a}y_{at}}{\leq 1,}{\ \forall \, a\in \mathcal{A} \label{1modeB1}}
	\addConstraint{x_p}{\in \{0,1\},}{\ \forall \, p \in \mathcal{P}} \label{xvarB1}
	\addConstraint{y_{at}}{\in \{0,1\},}{\ \forall \, t\in\mathcal{T}_a,\ \forall \, a \in \mathcal{A}.\label{fvarB1}} 
 \end{mini!} }
 
The objective \eqref{ObjB1} and constraints \eqref{1stB1} and \eqref{8thB1} function the same as the objective \eqref{ObjSM} and constraints \eqref{1stSM} and \eqref{8thSM}, respectively. Constraints \eqref{1modeB1} ensure the model selects at most one integer capacity to install on arc $a\in\mathcal{A}$.

While reformulation \eqref{rblpB1} is larger (with $\sum_{a\in\mathcal{A}}(|\mathcal{T}_a|-1)$ more variables and $|\mathcal{A}|$ more constraints) than \eqref{rblpSM}, we will show in Section \ref{sec:compForm} that a commercial optimization solver is able to produce stronger lower bounds when solving 
\eqref{rblpB1}.
However, it will also be shown that even small instances of these problems cannot be solved within the provided runtime and larger instances continue to have weak lower bounds despite the reformulation. Thus, in the next section, we will study a structured relaxation of the MCND problem to identify a new class of valid inequalities which can produce strong LP relaxations, and ultimately help solvers identify stronger lower bounds.

\section{Valid Inequalities} \label{sec:relax}

It is well known that MCND formulations may produce weak LP relaxations (referred to as LP(INT) and LP(BIN) for models \eqref{rblpSM} and \eqref{rblpB1}, respectively) for many practical instances due to the variable upper bound capacity constraints \eqref{8thSM} and \eqref{8thB1}. We provide an illustration of a root cause of this phenomenon in Figure \ref{fig:capConstr} by plotting 
the cost $f_a \tau_a$ of an arc $a$ in INT and LP(INT) given the volume assigned to flow on that arc, $
\sum_{k\in \mathcal{K}}\sum_{\{p\in \mathcal{P}_k|p \ni a\}}d_kx_p$. We see that the cost function for INT is a step function, whereas LP(INT) has a linear cost function with a slope of $\frac{f_a}{q_a}$ (due to continuous $\tau_a$ variables causing constraints \eqref{8thSM} of LP(INT) to hold at equality). Thus, LP(INT) minimizes cost by sending commodities along arcs (similarly, paths composed of a sequence of arcs) with the cheapest marginal cost per unit of demand sent. When the volume sent along an arc constitutes a small proportion of the integer capacity, LP(INT) is a weak approximation of INT. Note the same holds for BIN and LP(BIN).

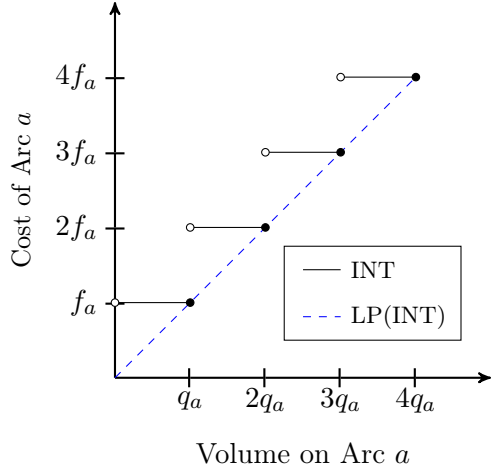
\begin{figure}[htbp]
\centering
\begin{tikzpicture}[x=1cm,y=1cm]

  \draw[-|,thick,>=stealth',shorten >=0pt] (0,0) -- (1,0) ;
  \draw[-|,thick,>=stealth',shorten >=0pt] (1,0) -- (2,0) ;
  \draw[-|,thick,>=stealth',shorten >=0pt] (2,0) -- (3,0) ;
  \draw[-|,thick,>=stealth',shorten >=0pt] (3,0) -- (4,0) ;
  \draw[->,thick,>=stealth',shorten >=0pt] (4,0) -- (5,0) ;
  \draw[-|,thick,>=stealth',shorten >=0pt] (0,0) -- (0,1) ;
  \draw[-|,thick,>=stealth',shorten >=0pt] (0,1) -- (0,2) ;
  \draw[-|,thick,>=stealth',shorten >=0pt] (0,2) -- (0,3) ;
  \draw[-|,thick,>=stealth',shorten >=0pt] (0,3) -- (0,4) ;
  \draw[->,thick,>=stealth',shorten >=0pt] (0,4) -- (0,5) ;
  \draw[domain=0:4,variable=\x,dashed,blue] plot ({\x},\x);

  \node[anchor=east] at (-0.1,1) {$f_a$};
  \node[anchor=east] at (-0.1,2) {$2f_a$};
  \node[anchor=east] at (-0.1,3) {$3f_a$};
  \node[anchor=east] at (-0.1,4) {$4f_a$};

  \node[] at (1,-0.3) {$q_a$};
  \node[] at (2,-0.3) {$2q_a$};
  \node[] at (3,-0.3) {$3q_a$};
  \node[] at (4,-0.3) {$4q_a$};

  \draw (0,1) -- (1,1); 
  \draw (1,2) -- (2,2);
  \draw (2,3) -- (3,3);
  \draw (3,4) -- (4,4);
  \draw[fill=white,draw=black] (0,1) circle [radius=0.5mm];
  \draw[fill=white,draw=black] (1,2) circle [radius=0.5mm];
  \draw[fill=white,draw=black] (2,3) circle [radius=0.5mm];
  \draw[fill=white,draw=black] (3,4) circle [radius=0.5mm];
  \draw[fill=black,draw=black] (1,1) circle [radius=0.5mm];
  \draw[fill=black,draw=black] (2,2) circle [radius=0.5mm];
  \draw[fill=black,draw=black] (3,3) circle [radius=0.5mm];
  \draw[fill=black,draw=black] (4,4) circle [radius=0.5mm];

  \node[] at (2.5,-1) {Volume on Arc $a$};
  \node[rotate=90] at (-1.25,2.5) {{\small Cost of Arc $a$}};

  \draw[] (2.25,1.75) rectangle (4.6,0.5);
  \draw[blue,dashed]  (2.5,0.8125) -- (3,0.8125); 
  \node[anchor=west] at (3,0.8125) {{\footnotesize LP(INT)}};
  \draw[] (2.5,1.4375) -- (3,1.4375); 
  \node[anchor=west] at (3,1.4375) {{\footnotesize INT}};

\end{tikzpicture}
\caption{Comparing the cost of volume flowing arc $a$ for INT and LP(INT).}
\label{fig:capConstr}
\end{figure}

A common way to strengthen the MCND LP relaxation is to add the following
disaggregated capacity linking constraints 
to LP(INT) and LP(BIN), respectively:
\begin{align}
    t_{ak}^{\min} \sum_{\{p\in \mathcal{P}_k|p \ni a\}}x_p &\leq \tau_a, \quad \forall \, a\in\mathcal{A}, \ \forall \, k\in\mathcal{K},\label{eqn:intDA} \\
    \sum_{\{p\in \mathcal{P}_k|p \ni a\}}x_p &\leq \sum_{t=t_{ak}^{\min}}^{T^{\max}_a}y_{at}, \quad \forall \, a\in\mathcal{A}, \ \forall \, k\in\mathcal{K},\label{eqn:binDA}
\end{align}
where $t_{ak}^{\min}$ is the minimum number of capacity units required to transport commodity $k\in\mathcal{K}$ on arc $a\in\mathcal{A}$ and is defined as $t_{ak}^{\min}=\left\lceil\frac{d_k}{q_a}\right\rceil$. 
If $t^{\min}_{ak}=1$ for all $ a\in\mathcal{A}$ and $ k\in\mathcal{K}$ (i.e., no individual commodity demand exceeds an arc's single-unit capacity), the solutions of LP(INT) and LP(BIN) with the addition of disaggregated capacity linking constraints \eqref{eqn:intDA} and \eqref{eqn:binDA}, respectively, are equivalent.
However, if this is not true and $t_{a'k'}^{\min}>1$ for some $a'\in \mathcal{A}$ and $k'\in \mathcal{K}$, then
LP(BIN) with constraints \eqref{eqn:binDA} can produce a marginally stronger relaxation solution than LP(INT) with constraints \eqref{eqn:intDA}. See 
\ref{app:intBinDA} 
for an example. 

Potentially stronger valid inequalities can be generated from simple structured relaxations over more complicated sets, like those focused on the unsplittable flow arc-set polyhedron (i.e., the
convex hull of solutions to capacity constraints of the form \eqref{8thSM} or \eqref{8thB1}) \cite{atamturk2002splittable,AlperChND,chen2021exact,bienstock1998minimum,brockmuller1996designing,van2002polyhedral}. Thus, in the following section, we investigate a simple structured relaxation of BIN \eqref{rblpB1} in which we also focus on the capacity-related constraints, including the single-capacity selection constraints \eqref{1modeB1}.
We will show through examples, and later \rev{show}
through a computational study, that the two new classes of valid inequalities developed can produce stronger LP relaxations as compared to when adding constraints \eqref{eqn:binDA} to LP(BIN).

\subsection{Valid Inequalities from Relaxation}
For a given arc $a\in \mathcal{A}$, we consider the following set:
\begin{eqnarray}\label{eq:set}
S:= \left\{ (x, y) \in \{0,1\}^{|\mathcal{P}_a|} \times \{0,1\}^{|\mathcal{T}_a|} \,\left |\, \sum_{p\in \mathcal{P}_a}d_px_p \leq q_a\sum_{t \in\mathcal{T}_a} t y_t ; \sum_{t\in \mathcal{T}_a}y_t\leq 1 \right. \right\},
\end{eqnarray}
where $\mathcal{P}_a = \{p \in \cup_{k\in\mathcal{K}} \mathcal{P}_k|p\ni a\}$, $d_p \in \mathbb{Z}_{+}$ for $p \in \mathcal{P}_a$ and $q_a \in \mathbb{Z}_{+}$. Note that a path $p\in\mathcal{P}_k$ is unique for commodity $k$; thus, to simplify notation, we use $d_p$ in place of $d_k$ for $p\in\mathcal{P}_k$.

We first analyze the convex hull solutions of \eqref{eq:set}, from which we deduce the general form of the constraints defining the set.

\begin{proposition}\label{prop:Nonneg}
There exists non-negative matrices $G, H$ such that 
$$\textup{conv}(S) = \left\{ (x, y) \in [0,\ 1]^{|\mathcal{P}_a|} \times [0,\ 1]^{|\mathcal{T}_a|} \,\left |\, Gx \leq Hy; \sum_{t\in\mathcal{T}_a}y_t\leq 1 \right. \right\}.$$
\end{proposition}
\begin{proof}
Observe that if $y_t = 0$ for all $t \in \mathcal{T}_a$, then $x = 0$. Otherwise for $t \in \mathcal{T}_a$, let: 
$$S_t:= \left\{ x \in \{0,1\}^{|\mathcal{P}_a|}  \, \left|\, \sum_{p\in\mathcal{P}_a}d_p x_p \leq q_at\right. \right\}, $$ and let
$$\textup{conv}(S_t) = \left\{u \geq 0 \,\left|\, s_a^tu \leq r^t \right.\right\}.$$ 
Then it is clear that,
$$S = (\textbf{0}, \textbf{0}) \cup \bigcup_{t \in \mathcal{T}_a} (S_t, e_t).$$
Therefore, we obtain that 
\begin{eqnarray*}
\textup{conv}(S) &=& \textup{proj}_{x, y}(T), \ \textup{where}\\
T &=& \left\{(x, y, x^1, x^2, \dots, x^{T^{\max}_a}) \,\left|\, x^t \geq 0, s_a^tx^t - r^ty_t \leq 0, t \in \mathcal{T}_a; x - \sum_{t \in\mathcal{T}_a} x^t =0; \sum_{t\in \mathcal{T}_a}y_t\leq 1  \right.\right\}.
\end{eqnarray*}
We may apply Fourier-Motzkin elimination to $T$ to project out the variables $x^1, \dots, x^{T^{\max}_a}$. Since these variables only appear in the constraints $x^t \geq 0, s_a^tx^t - r^ty_t \leq 0, t \in \mathcal{T}_a; x - \sum_{t \in\mathcal{T}_a} x^t =0$, where the right-hand-side is $0$ for all the constraints, we arrive at the fact that (after Fourier-Motzkin elimination)
$$\textup{conv}(S) = \left\{ (x, y) \in [0,\ 1]^{|\mathcal{P}_a|} \times [0,\ 1]^{|\mathcal{T}_a|} \,\left |\, \widetilde{G}x \leq \widetilde{H}y; \sum_{t\in\mathcal{T}_a}y_t\leq 1 \right. \right\}.$$
It remains to prove that $\widetilde{G}$ and $\widetilde{H}$ are non-negative. 

We will first consider $\widetilde{G}$. Suppose that $\widetilde{G}_{v,w} <0$ for some $(v,w)$.  We claim that if we replace $\widetilde{G}_{v,w}$ with $0$, the resulting inequality:
\begin{eqnarray}\sum_{p \in \mathcal{P}_a\setminus \{w\}}\widetilde{G}_{v,p}x_p \leq \widetilde{H}_vy\label{eq:updeq}
\end{eqnarray} is still valid for $S$. This will prove the result, since replacing $\widetilde{G}_{v,w} $ by $0$ results in a stronger inequality that dominates the original inequality (since $x_w\geq 0$). Suppose $(\hat{x}, \hat{y}) \in S$  and this point does not satisfy (\ref{eq:updeq}). First note that $\hat{x}_w = 1$, since if $\hat{x}_w = 0$, then  $\sum_{p \in \mathcal{P}_a\setminus \{w\}}\widetilde{G}_{v,p}\hat{x}_p = \sum_{p \in \mathcal{P}_a}\widetilde{G}_{v,p}\hat{x}_p$. However, now observe that $(\widetilde{x}, \hat{y}) \in S$, where:
\begin{eqnarray*}
\widetilde{x}_p = \left\{\begin{array}{rl} \hat{x}_p, & p \neq w\\ 0, & p = w \end{array}. \right.
\end{eqnarray*}
Therefore, $\widetilde{H}_v\hat{y} \geq  \sum_{p \in \mathcal{P}_a}\widetilde{G}_{v,p}\widetilde{x}_p = \sum_{p \in \mathcal{P}_a\setminus \{w\}}\widetilde{G}_{v,p}\widetilde{x}_p= \sum_{p \in \mathcal{P}_a\setminus \{w\}}\widetilde{G}_{v,p}\hat{x}_p$. Thus, $(\hat{x}, \hat{y}) $  satisfies (\ref{eq:updeq}), a contradiction.

We now verify that $\widetilde{H}$ is non-negative. Note that $(\textbf{0}, e_p) \in S$ for all $p \in \mathcal{P}_a$. Therefore, if $H_{v,k} < 0$, we have that the point $(\textbf{0}, e_k)$ will not satisfy the inequality $\sum_{p \in \mathcal{P}_a}\widetilde{G}_{v,p}x_p \leq \widetilde{H}_vy$. This concludes the proof. 
\end{proof}

It is important to note that the general form of the constraints lacks a constant term in $Gx\leq Hy$, unlike the inequality introduced in \cite{chen2021exact}. The reasons for this are twofold: (i) arcs do not have pre-existing capacities, which would introduce a constant term to the right-hand side of \eqref{8thB1}, and (ii) our reformulation imposes that each arc has a uniquely defined capacity value such that only one $y_{t}>0$ for all $t\in\mathcal{T}_a$. If either condition were not true, a constant term would be required. \rev{Under these assumptions, the feasible region is monotone and therefore feasibility is preserved when additional capacity is installed (increasing $y$) or when commodities are removed (decreasing $x$). 
As a result, all inequalities describing $\text{conv}(S)$ can be written with non-negative coefficients. 
If pre-installed capacity or multiple simultaneous capacity levels were allowed, this monotonic structure would no longer hold, and the resulting inequalities could include constants or even negative coefficients.  
}

In the following sections, we describe new valid inequalities for $S$ that \rev{are derived from the structural form identified in}
Proposition \ref{prop:Nonneg}. 

\subsubsection{Single-Arc Commodity Packing Constraints}\label{sec:01VI}

Since $G$ is a non-negative matrix, we first explore a subset of constraints which involve only $\{0,1\}$ coefficients for the path variables $x$. 
In this case, since both $x$ and $y$ are binary variables, each element in $H$ is bounded above by $|\mathcal{P}_a|$, as the left-hand side cannot exceed this value and at most one $y$ variable on the right-hand side can be nonzero. 

For a given binary vector $g$ corresponding to the \rev{coefficients of the} $x$ \rev{variables}, \rev{the coefficient} $h_t$ for $y_t$ \rev{can be determined} by counting the maximum number of \textit{commodities} with at least one path selected in $g$ (i.e., $x$ variables \rev{with $g_p=1$}) that can be transported using, or \textit{packed} into, $t$ capacity units. We emphasize the number of commodities \rev{rather than} the number of paths, \rev{since} commodities \rev{may have multiple feasible} paths in $\mathcal{P}_a$ but should not be counted more than once.  

\rev{The inequalities derived here are closely related to the cover inequalities for the 0-1 knapsack polytope \citep{balas1978facets,wolsey1975faces}.
In our setting, the left-hand side represents the number of commodities included in the packing relation, while the coefficients on the right-hand side capture the maximum number of such commodities that can fit within $tq_a$ capacity units. 
}

We state this result formally next.

\begin{proposition}\label{prop:alphas}
\rev{For a given arc $a\in\mathcal{A}$,} consider a non-zero 
binary vector $g\;\rev{\in\{0,1\}^{|\mathcal P_a|}}$. 
The inequality
\begin{eqnarray}\label{eq:alpha}\sum_{p \in \mathcal{P}_a} g_px_p \leq \sum_{t \in \mathcal{T}_a} \alpha_t y_t,
\end{eqnarray}
is a valid inequality for $S$ if
\begin{eqnarray}\label{def:alpha}\alpha_t = \textup{max}\left\{\sum_{p \in \mathcal{Z}} \rev{g_p} x_p\,\left |\, \sum_{p \in \mathcal{Z}}d_px_p \leq tq_a,\; \rev{\sum_{p\in\mathcal{Z}_k}x_p\leq 1,}\; x_p \in \{0, 1\} \right. \right\},
\end{eqnarray}
where  $\mathcal{Z} = \{ p \in \mathcal{P}_a \,|\, g_p=1 \}$ \rev{ and $\mathcal{Z}_k=\mathcal{Z}\cap \mathcal{P}_k$ for each commodity $k$}. 
\end{proposition}
\begin{proof}
    Clearly (\ref{eq:alpha}) is satisfied by $(0, 0)$. Now consider any non-zero solution
    $(\hat{x}, \hat{y}) \in S$ with  $\hat{y}_t = 1$ for $t = \hat{t}$ and $\hat{y}_t =0$ for $t\neq \hat{t}$.
    Then, by the definition of $S$, we have that 
    $\sum_{p \in \mathcal{P}_a} d_p\hat{x}_p  \leq \hat{t}q_a.$ Therefore, we obtain that $\sum_{p \in \mathcal{Z}} d_p\hat{x}_p  \leq \sum_{p \in \mathcal{P}_a} d_p\hat{x}_p \leq \hat{t}q_a.$ Now by the definition of $\alpha_{t}$ in (\ref{def:alpha}), we have that $\alpha_{\hat{t}} \geq \sum_{p \in \mathcal{Z}}\rev{g_p}\hat{x}_p.$ In other words, we obtain that
    $$\sum_{p\in \mathcal{P}_a}g_p\hat{x}_p = \sum_{p \in \mathcal{Z}}\rev{g_p}\hat{x}_p \leq \alpha_{\hat{t}}  = \sum_{t \in \mathcal{T}_a}\alpha_t \hat{y}_t,$$
    that is (\ref{eq:alpha}) is satisfied by 
    $(\hat{x}, \hat{y})$. 
\end{proof}

\rev{Although $g_p=1$ for all $p\in\mathcal Z$ for the binary case, we keep the coefficient $g_p$ in the notation for consistency with the generalized case discussed later.}

Using Proposition \ref{prop:alphas},
we can define the following \textit{single-arc commodity packing} (SAC-Pack) constraints:
\begin{equation}\label{eq:01valid}
    \sum_{p\in\mathcal{Z}}x_p\leq \sum_{t\in\mathcal{T}_a}\alpha_ty_t, \quad \forall \, a\in\mathcal{A},
\end{equation}
where $\alpha_t$ is an integer coefficient 
equal to the maximum number of selected commodities (those with at least one path $p\in\mathcal{Z}$) that can be transported by $t\in\mathcal{T}_a$ capacity units. In a sense, constraints \eqref{eq:01valid} are a ``smart re-aggregation" of linking constraints \eqref{eqn:binDA} for a set of commodities with at least one path $p\in\mathcal{Z}$, and can at times even dominate them.

\begin{example}
    Consider the following example where we \rev{illustrate}
    how SAC-Pack constraints \eqref{eq:01valid} can dominate disaggregated linking constraints \eqref{eqn:binDA}.

    Assume $\mathcal{P}_a=\mathcal{Z}=\{1,2\}$. Let the solution of LP(BIN) be $x_1=0.5$ and $x_2=0.4$ and let $d_1=60$ and $d_2=70$ with $q_a=100$. We will also assume that $T^{\max}_a = 2$ for clarity of exposition. Constraints \eqref{8thB1} set the capacity of arc $a$ as $y_{a1} = 0.58$ units (similarly, $y_{a1} + 2y_{a2} = 0.58$). To strengthen LP(BIN), we add the following constraints \eqref{eq:01valid}:
    $$x_1+x_2\leq y_{a1}+2y_{a2}.$$
    This increases the installed capacity to $0.9$, whereas if we only add the following constraints \eqref{eqn:binDA}:
    \begin{align*}
        x_1&\leq y_{a1} + y_{a2},\\
        x_2&\leq y_{a1} + y_{a2},
    \end{align*}
    $y_{a1}=0.58$ satisfies both constraints, and the solution of LP(BIN) does not change.
    
\end{example}

We note that the process of determining the coefficients given by (\ref{def:alpha}) can be done efficiently by sorting the selected commodities in order of non-decreasing demand values, adding demands until the capacity is exceeded, and setting the coefficient to the number of demands added minus one (as the last demand cannot be transported). \rev{A dynamic program (DP) can also be used to compute these coefficients, but for the SAC-Pack constraints we found that a separation IP solved faster in our experiments. Such a DP becomes more useful in the more general setting discussed later.} We next formally define this \rev{IP-based} separation routine to identify violated constraints \eqref{eq:01valid}.

\paragraph{Separation}
Given the set $S$ described in \eqref{eq:set} and a point $(x^*,y^*)$, there exists a SAC-Pack constraint \eqref{eq:01valid} violated by $(x^*,y^*)$ if and only if there exists a set $\mathcal{Z}\subseteq \mathcal{P}_a$ such that $\sum_{p\in\mathcal{Z}}x_p^* - \sum_{t\in\mathcal{T}_a}\alpha_ty_{at}^* > 0$. Because we can determine the $\alpha$ values in closed form (see Proposition \ref{prop:alphas}), we can explicitly check for such a violation by formulating the problem as an IP with the objective $\max_{\mathcal{Z}\subseteq\mathcal{P}_a}\{\sum_{p\in\mathcal{P}_a}z_px_p^* - \sum_{t\in\mathcal{T}_a}\alpha_ty_{at}^*\}$.
In this formulation, binary variables $z_p$ indicate whether path $p$ is selected to be in $\mathcal{Z}$ and integer variable $\alpha_t$ represents the number of selected commodities that can be transported by $t\in\mathcal{T}_a$ capacity units. 
If the objective value is greater than 0, we add the violated constraint to the model.

\begin{table}[!htb]
\footnotesize
\centering
\caption{SAC-Pack separation problem variable definitions.}
\begin{tabular}{lp{11cm}}
\hline
\textbf{Variable} 	& \textbf{Description}\\
\hline
$z_{p}\in \{0,1\}$ & Indicate whether path $p \in \mathcal{P}_a$ is selected for set $\mathcal{Z}\subseteq \mathcal{P}_a$.\\
$\alpha_t\in \Z_{\geq0}$ & Coefficient of $y_t$ representing the number of selected commodities that can be transported by $t\in\mathcal{T}_a$ capacity units.\\
$s_{kt}\in \{0,1\}$ & Indicate commodity $k\in\mathcal{K}_a$ is in the maximum subset that can be transported by $t\in\mathcal{T}_a$ capacity units.\\
$u_t\in \{0,1\}$ & Indicate if all selected commodities can be transported by $t\in\mathcal{T}_a$ capacity units (deactivates constraints).\\
$v_t\in \{0,1\}$ & Indicate if not all selected commodities can be transported by $t\in\mathcal{T}_a$ capacity units (activates constraints).\\
$w_{kt}\in \{0,1\}$ & Indicate if a selected commodity $k\in\mathcal{K}_a$ has at least one path selected but not in the subset of commodities that can be transported by $t\in\mathcal{T}_a$ capacity units.\\
\hline
\end{tabular}
\label{tbl:variablesSep} 
\end{table}

Let the set $\mathcal{K}_a\subseteq \mathcal{K}$ be defined as the set of commodities with at least one path containing arc $a$ (or $\{k\in\mathcal{K}||\{p\in \mathcal{P}_k|p\ni a\}|\geq 1\}$)
in order of non-decreasing demands.
We define the decision variables for the separation problem in Table \ref{tbl:variablesSep}.
We can now formulate the separation problem as: 

{\small
\begin{maxi!}[2]
	{}{\sum_{p \in \mathcal{P}_a} z_px^*_p - \sum_{t\in \mathcal{T}_a}\alpha_ty^*_{at}}{\label{sep}}{\label{ObjSep}}
	\addConstraint{\alpha_t \geq }{\sum_{k\in \mathcal{K}_a}s_{kt} - 1 + u_t,}{\  \forall \, t \in \mathcal{T}_a \label{1stSep}}
	\addConstraint{s_{kt} \leq }{\sum_{\{p\in \mathcal{P}_k|p \ni a\}} z_p,}{  \ \forall \, t \in \mathcal{T}_a, \, \forall \, k\in \mathcal{K}_a  \label{2ndSep}}
	\addConstraint{t(1-u_t)+\epsilon}{\leq \sum_{k \in \mathcal{K}_a} \frac{d_k}{q_a} s_{kt},}{\ \forall \, t \in \mathcal{T}_a \label{3rdSep}}
	\addConstraint{z_p - s_{kt}}{\leq w_{kt},}{  \ \forall \,  t \in \mathcal{T}_a, \, \forall \, k\in \mathcal{K}_a, \, \forall \, p \in \{\mathcal{P}_{a}\cap \mathcal{P}_k\}\label{4thSep}}
	\addConstraint{\sum_{k\in\mathcal{K}_a}w_{kt}}{\leq |\mathcal{K}_a|v_t,}{\ \forall \, t \in \mathcal{T}_a} \label{5thSep}
        \addConstraint{u_t}{\leq 1-v_t,}{\ \forall \, t \in \mathcal{T}_a \label{6thSep}}
        \addConstraint{\left|\mathcal{K}_a\setminus \{k\leq k_1\}\right|(s_{k_1t}+1-w_{k_1t})}{\geq \sum_{\{k_2 \in \mathcal{K}_a|k_2>k_1\}} s_{k_2t},\quad}{\ \forall \, t \in \mathcal{T}_a, \, \forall \, k_1\in \mathcal{K}_a\label{7thSep}}
	\addConstraint{z_{p}}{\in \{0,1\},}{\ \forall \, p \in \mathcal{P}_a} 
        \addConstraint{u_t\in \{0,1\},}{\ v_t\in \{0,1\}, \ \alpha_t \in \Z_{\geq0},}{\ \forall \, t \in \mathcal{T}_a} 
        \addConstraint{s_{kt}\in \{0,1\},}{\ w_{kt}\in \{0,1\},}{\ \forall \, k_1\in \mathcal{K}_a, \, \forall \, t \in \mathcal{T}_a.} 
 \end{maxi!} }
 
Constraints \eqref{1stSep} capture the value of each $\alpha_t$ coefficient by summing the commodities selected for the subset. The 1 is subtracted because the $s_{kt}$ variables are collected until the total demand exceeds $t$ capacity units; thus, the last commodity added should be removed to make it feasible. However, if the total volume of the commodities is less than $t$ capacity units, all commodities should be counted; thus, we add the binary variable $u_t$ which is equal to 1 when this is the case. In the case where all commodities have $s_{kt}=1$ but the total volume exceeds $t$, $u_t$ can and still will equal $0$ because the objective minimizes the $\alpha_t$ coefficients. Constraints \eqref{2ndSep} enforce that commodities cannot be selected for the maximum subset that fits in $t$ capacity units if none of its paths were selected. Constraints \eqref{3rdSep} determine the number of commodities plus 1 that can be transported by $t$ capacity units. The constraints are turned off if the total commodity volume is less than $t$ capacity units. The $\epsilon$ parameter ensures that the right-hand side does not equal $t$, as we later subtract 1, assuming that the volume is exceeded. Thus, we note that for any $\epsilon>0$, the generated inequality is valid, with larger values of $\epsilon$ possibly making the resulting inequality weaker. 
In our computational study, we set $\epsilon=0.001$, which is sufficiently small to avoid weakening the constraints
because $\epsilon<\min_{k\in\mathcal{K}_a}\{d_k/q_a\}$. Constraints \eqref{4thSep} determine if at least one path for commodity $k$ was selected but is not included in the subset of commodities that can be transported by $t$ capacity units. Note that here we use $\{\mathcal{P}_a\cap \mathcal{P}_k\}$ to represent the set of paths for commodity $k\in\mathcal{K}_a$ that contain arc $a$. Constraints \eqref{5thSep} set $v_t = 1$ if any $w_{kt} = 1$ as set by constraints \eqref{4thSep}. Constraints \eqref{6thSep} ensure that $u_t=0$ if the size of the subset of commodities that can be transported by $t$ capacity units, plus 1, is less than the total number of commodities whose paths were selected. Finally, constraints \eqref{7thSep} are precedence constraints for each capacity $t$ that ensure commodities are selected in order of non-decreasing demands, if selected at all.

To reduce the size of the separation problem, one can consider only the paths and capacity variables with non-zero values in $(x^*,y^*)$. The coefficients for variables $y^*_t = 0$ can be calculated once the paths are selected for the constraint as previously described. Another way to reduce the problem size is to aggregate paths for the same commodity. That is, let $z_k$ be a binary variable that indicates if all non-zero paths in $\mathcal{P}_a$ for commodity $k\in\mathcal{K}_a$ are selected for the set $\mathcal{Z}$. The objective changes to $\sum_{k\in\mathcal{K}_a}z_k\sum_{p\in\{\mathcal{P}_a\cap\mathcal{P}_k\}}x^*_p-\sum_{t\in \mathcal{T}_a}\alpha_ty^*_{at}$, the number of variables is reduced by $|\mathcal{P}_a|-|\mathcal{K}_a|$, and the number of constraints is reduced by $|\{\mathcal{P}_a\cap\mathcal{P}_k\}|-1$ for each $k\in\mathcal{K}_a$. 

\subsubsection{Generalized Single-Arc Commodity Packing Constraints}\label{sec:intVI}
We now investigate the case in which the elements of $G$ (or coefficients of the $x$  variables) 
\rev{can take integer values in} $\{0, 1, \dots, B\}$. 
\rev{The result in Proposition~\ref{prop:alphas} applies directly to this generalized setting, where $g_p \in \{0,1,\dots,B\}$ represents integer coefficients for the path variables. The proof remains identical, as the argument does not depend on the specific values of $g_p$ beyond their nonnegativity. 
We next discuss the usefulness of these generalized inequalities and how they can strengthen the binary case through}
the following example.
\begin{example}
    In this example, we \rev{illustrate}
    how constraints \eqref{eq:01valid} can be strengthened to cut off solutions in some directions if we allow $x$ variables to have coefficients in the list $\{0, 1, \dots, B\}$.

    Let $q_a=100$ and $\mathcal{Z}=\{1,2,3,4\}$ with $d_1=30,\, d_2=30,\, d_3=30,$ and $d_4=60$, where each $p\in\mathcal{Z}$ is for a unique commodity. Assuming $T^{\max}_a=2$, constraints \eqref{eq:01valid} take the form:
    $$x_1+x_2+x_3+x_4\leq 3y_{a1}+4y_{a2}.$$
    Here, the three smaller commodities can be transported using one unit of capacity, whereas transporting all four requires two units. 
    If $x$ variables had coefficients in the list $\{0, 1, \dots, B\}$, 
    we could strengthen the previous inequality for the single-unit capacity case (i.e., when $y_{a1}>0$) by increasing the coefficient for commodity 4:
    $$x_1+x_2+x_3+2x_4\leq 3y_{a1}+5y_{a2}.$$
    \rev{This adjustment reflects the fact that commodity 4} can only be transported with at most one other commodity when only one unit of capacity is installed. 
    Notice that we also increased the coefficient for $y_{a2}$ by 1; thus, when two units of capacity are installed, the constraints are equivalent in that they allow for all four commodities to be transported.

    To \rev{illustrate}
    how some solutions may be cut off, let the solution of LP(BIN) be $x_1=x_2=x_3=x_4=0.5$, setting $y_{a1}=0.75$. This solution remains feasible with constraints \eqref{eq:01valid}, as the left-hand side is $2$ and the right-hand side is $2.25$. However, if you increase the coefficient of $x_4$ to $2$, this solution is no longer feasible, as the left-hand side is now $2.5$. Thus, $y_{a1}$ is increased to $0.83$.
    
\end{example}

\rev{Following Proposition \ref{prop:alphas},} the \textit{generalized single-arc commodity packing} (Gen-SAC-Pack) constraints \rev{take the form}:
\begin{equation}\label{eq:intValid}
    \sum_{p\in\mathcal{Z}}\theta_px_p\leq \sum_{t\in\mathcal{T}_a}\alpha_ty_t, \quad \forall \, a\in\mathcal{A}, 
\end{equation}
where $\theta_p$ are non-negative integer coefficients for paths $p\in\mathcal{Z}$ bounded above by $B$. 

\rev{We use two methods to identify Gen-SAC-Pack constraints \eqref{eq:intValid}: a post-processing procedure that strengthens existing SAC-Pack constraints and a separation approach that generates new Gen-SAC-Pack inequalities directly.}

\paragraph{Post-processing SAC-Pack constraints}
A simple method to obtain constraints of the form \eqref{eq:intValid} is to post-process the SAC-Pack constraints found using \eqref{sep}. The idea is to identify paths (or commodities) which cannot be transported with $\alpha_t -1$ other commodities for $t = \min\{t\in\mathcal{T}_a\,|\,y_t > 0\}$. Recall the example above, the largest commodity demand (represented by $x_4$) could not be transported with two other commodities using only one unit of capacity; thus, its coefficient was increased by one. This idea can also be extended to more than one path with an increased coefficient. 
\rev{Coefficient increases of this kind are related to techniques in the general literature on lifting and strengthening cover inequalities \citep{marchand2002cutting} as well as to the coefficient-increase procedures developed by \citeauthor{dietrich1992tightening}, \citeyear{dietrich1992tightening}.} 
In Section \ref{sec:comp}, we will \rev{show}
 that this simple approach is able to further strengthen LP(BIN) with constraints \eqref{eq:01valid} with a marginal increase in runtime.

\paragraph{Separation} To obtain potentially stronger cuts, we use the following separation approach to define new Gen-Sack-Pack constraints of the form \eqref{eq:intValid}.

We are given the set $S$ described in (\ref{eq:set}) and a point $(x^*, y^*)$ that we want to separate from the convex hull of $S$. We obtain separating inequalities using the following row-generation approach \citep{boyd1994fenchel}.
\begin{enumerate}
\item Let $\{x^j, y^j\}_{j = 1}^r$ be a subset of points in $S$. 
\item Solve: 
\begin{eqnarray}\label{eq:SEPmaster}
\begin{array}{rcl}
\textup{SEPVAL}:= &\textup{max}_{\theta, \alpha} & \theta^{\top} x^* - \alpha^{\top}y^*\\
&\textup{s.t.}& \theta^{\top} x^j - \alpha^{\top}y^j \leq 0, \ \forall \ j \in [r], \\
&& \alpha_t \leq \alpha_{t +1}, \ \forall \ t \in \mathcal{T}_a\setminus\{T^{\max}_a\},\\
&& \theta_p \leq B, \ \forall \ p \in \mathcal{P}_a, \\
&& \theta\in \Z_{\geq 0},\ \alpha \in \Z_{\geq 0}.  \\
\end{array}
\end{eqnarray}
If $\textup{SEPVAL} \leq 0$, then STOP as there is no separating inequality. Else, let $(\theta^*, \alpha^*)$ be an optimal solution of (\ref{eq:SEPmaster}).
\item Solve:
\begin{eqnarray}\label{eq:SEPchild}
\begin{array}{rcl}
\textup{FEASVAL}:= &\textup{max}_{x,y}& \theta^{* \top} x - \alpha^{* \top}y\\
&\textup{s.t.}& (x,y) \in S.
\end{array}
\end{eqnarray}
If $\textup{FEASVAL} = 0$\footnote{Note $\textup{FEASVAL} \geq 0$ since $(0,0)\in S$}, then STOP as $\theta^{*\top}x \leq \alpha^{*\top}y$ is a separating inequality. Else, let $(\hat{x}, \hat{y})$ be an optimal solution of (\ref{eq:SEPchild}). Let $r \leftarrow r +1$, let $(x^r, y^r) \leftarrow (\hat{x}, \hat{y})$ and return to STEP 2 above.  
\end{enumerate}
Note that parameter $B$ in \eqref{eq:SEPmaster} above is an artificial upper bound. In Section \ref{sec:comp}, we will \rev{show}
that using $B \leq 3$ gives high-quality cuts in comparison to larger upper bounds. As in the separation problem defined in Section \ref{sec:01VI}, paths for the same commodity can be aggregated into a single variable to help reduce required compute time.  See 
\ref{app:dp} 
for the dynamic programming (DP) approach we use in our computational experiments to further decrease the runtime of solving \eqref{eq:SEPchild}. 
While the previously-defined row-generation approach \rev{and DP} with $B=1$ can also be used to identify SAC-Pack constraints \eqref{eq:01valid}, it requires more compute time as compared to solving model \eqref{sep}. Thus, we only employ this separation method when $B>1$. 

We also note that while we focus on the \textit{path-based}, unsplittable MCND problem in this work, the previous valid inequalities \eqref{eq:01valid} and \eqref{eq:intValid} and their separation routines can also be applied to \textit{arc-based}, unsplittable MCND problems. In this case, the set $\mathcal{Z}$ would consist of the commodity-specific binary arc-flow variables. 
We \rev{show}
this later in Section \ref{sec:comp} by generating valid inequalities \eqref{eq:01valid} and \eqref{eq:intValid} for arc-based, unsplittable MCND models for the Canad instances \citep{CanadInst}.

\section{\rev{Algorithm Engineering}}  \label{sec:helper}
\rev{In this section, we develop an algorithmic framework that combines our SAC-Pack and Gen-SAC-Pack inequalities with a procedure for generating and selecting metric inequalities via Lagrangian relaxation. We begin by introducing the metric inequalities used in our approach and how they are constructed, then describe the Lagrangian selection mechanism.
The section concludes by combining these elements into a framework we will test  in our computational study.}

\subsection{\rev{Metric Inequalities}}
\rev{We begin by revisiting metric inequalities, which are well-known valid inequalities for the multicommodity flow polytope \citep{AlperChND, bienstock1998minimum, costa2009benders}. Metric inequalities arise from the dual of a multicommodity flow formulation and capture, in a single constraint, the aggregate capacity required to support all simultaneous commodity flows. By assigning non-negative weights to the arcs, one defines a distance metric on the network; under this metric, the distance between any two nodes is the length of the shortest path connecting them. A metric inequality asserts that the total weighted capacity of the network must be at least as large as the sum of each commodity’s demand multiplied by the distance from its origin to its destination. Consequently, these inequalities capture global structural constraints, generalizing classical min-cut/max-flow bounds, that any feasible multicommodity flow must satisfy.

Metric inequalities are redundant when added to our BIN formulation. However, we observe in computational tests that including a selected set of metric inequalities constructed from an LP-feasible capacity vector leads to stronger dual bounds and improved overall solver performance on our MCND instances.\footnote{To show that the benefits are not specific to one commercial solver, we also provide limited results for two other solvers in 
\ref{app:solvers}.} This behavior is consistent with findings in the literature showing that certain structured aggregations can enhance LP-based bounding procedures \citep{bodur2018aggregation}. Motivated by these observations, we incorporate metric inequalities as “helper” inequalities within our overall strategy.}

To generate metric inequalities, we use a feasible capacity vector $\Bar{q}$ defined by the optimal solution $(\Bar{x},\Bar{y})$ to LP(BIN)
(or $\Bar{q}_a=q_a\sum_{t\in\mathcal{T}_a}t\Bar{y}_{at},\ \forall \, a \in \mathcal{A}$).
We then use the standard linear programming dual representation of the multicommodity flow polytope \eqref{odMI}
to generate individual metric inequalities. Bounding any extreme ray of the dual creates an associated inequality. Note that this dual is formulated using the arc-based relaxation of the original multicommodity flow problem (i.e., use commodity-specific arc-flow variables as opposed to path-flow variables). 
Since the arc-based formulation is a relaxation, the identified metric inequalities remain valid for the path-based formulation. Importantly, 
the capacity vector $\Bar{q}$ used to generate the metric inequalities is feasible for the more restricted path-based formulation by definition. Moreover, the resulting metric inequalities are directly applicable to the path-based formulation. \rev{This is because} the arc-flow variables are projected out given $\Bar{q}$, \rev{leaving inequalities that depend} only on arc-capacity variables and the commodity demands transiting those arcs.

We further reduce the problem size by \rev{aggregating commodities either by origin or by destination. Under origin-based aggregation,} all commodities with the same origin node are combined into a single commodity; \rev{under destination-based aggregation,} commodities are combined according to their destination node. These aggregations are feasible because the demand flow in the LP is splittable.

We now outline the steps to generate a single helper metric inequality and integral metric inequality (i.e., user cut) when commodities are aggregated by origin.

\paragraph{Origin-based commodity approach}
Let all commodities originating from the same origin $o(k)=o$ be consolidated into one commodity $o\in\mathcal{O}$, where the set $\mathcal{O}$ represents the redefined commodity set. We define the path set for commodity $o\in\mathcal{O}$ as $\mathcal{P}_o=\{\cup_{\{k\in \mathcal{K}|o(k)=o\}}\mathcal{P}_k\}$.
Commodity demand $d_k$ for $k\in\mathcal{K}$ is relabeled as $\psi^o_d$ \rev{for each commodity} originating at $o(k)=o$ and destined for $d(k)=d$.
\rev{These destination-specific demands are included in the total demand emanating from origin $o$, given by} $\psi^o_o=\sum_{\{k\in \mathcal{K}|o(k)=o\}}-d_k$.
Note $\psi^o_i=0$ when $i\neq o(k),d(k),
$ for all $\{k\in \mathcal{K}|o(k)=o\}$.
     
     Given the optimal solution $(\Bar{x},\Bar{y})$ to the (path-based) LP(BIN), set capacities 
     to the values  $\Bar{q}_a=q_a\sum_{t\in\mathcal{T}_a}t\Bar{y}_{at}$ for each $a \in \mathcal{A}$.
     Formulate the dual of the LP relaxation of the \textit{arc-based} MCND problem using dual variables $u$ and $v$, where $u$ corresponds to the arc flow conservation constraints\footnote{See 
     \ref{app:arcBased}.}
     and $v$ corresponds to the variable upper bound capacity constraints (i.e., arc-based versions of \eqref{8thB1}).
     Also note that because one flow conservation constraint for each commodity $o\in\mathcal{O}$ is redundant, we may assume $u^o_o=0$ for each $o\in\mathcal{O}$. We now formulate the dual as follows:
     {\small \begin{maxi!}[2]
    	{}{\sum_{i\in \mathcal{N}}\sum_{o\in \mathcal{O}_i}\psi^o_iu_i^o
     - \sum_{a \in \mathcal{A}}\Bar{q}_av_a }{\label{odMI}}{\label{ObjSMCap}}
        \addConstraint{u^o_j-u^o_i}{\leq v_{ij},}{\  \forall \,  o \in \mathcal{O}_{ij}, \ \forall \ i,j \in \mathcal{A} \label{arcDC}}
        \addConstraint{u^o_o}{=0,}{  \ \forall \,  o\in \mathcal{O} \label{origDC}}
        \addConstraint{v_a}{\in \mathbb{R}_{\geq 0},}{\ \forall \, a \in \mathcal{A}} \label{vvar}
        \addConstraint{u^o_i}{\in \mathbb{R},}{\ \forall \, o \in \mathcal{O}_i, \ \forall \ i \in \mathcal{N}.\label{wvar}}
    \end{maxi!} }
    The set $\mathcal{O}_{ij}$ represents the commodities that have at least one path using arc $(i,j)\in\mathcal{A}$ (i.e., $\mathcal{O}_{ij}=\{o\in \mathcal{O} | (i,j)\in \mathcal{P}_o\}$). Similarly, $\mathcal{O}_i$ denotes the commodities with one or more paths that include node $i\in\mathcal{N}$ (i.e., $\mathcal{O}_{i}=\{o\in \mathcal{O} | i\in \mathcal{P}_o\}$).
    As noted earlier, this is a relaxed version of the original path-based model because path set $\mathcal{P}_o=\{\cup_{\{k\in \mathcal{K}|o(k)=o\}}\mathcal{P}_k\}$ is used to define $\mathcal{O}_{ij}$. Consequently, a commodity $k$ with $o(k)=o$ is allowed to traverse any arc in $\mathcal{P}_o$, including arcs that do not appear in its individual path set $\mathcal{P}_k$.
    
    In an iterative fashion, we optimize \eqref{odMI} with $v_a=1$ for each arc transporting commodity demands (i.e., the set of arcs $\{a\in\mathcal{A}|\sum_{t\in\mathcal{T}_a}\Bar{y}_{at}>0\}$).
    In each iteration, if the optimal objective value is 0, let $(\Bar{v},\Bar{u})$ be an optimal solution to \eqref{odMI} and collect the following helper metric inequality:
    \begin{align}\label{eq:MIbase}
        \sum_{a\in \mathcal{A}}\Bar{v}_aq_a\sum_{t\in\mathcal{T}_a}ty_{at}  &\geq  \sum_{i\in \mathcal{N}}\sum_{o \in \mathcal{O}_i}\psi^o_i\Bar{u}^o_i
    \end{align}
    and \textit{integral metric inequality} generated from constraints \eqref{eq:MIbase}:
    \begin{align}\label{eq:intMI}
        \sum_{a\in \mathcal{A}}\Bar{v}_a\left\lceil \frac{q_a}{q_{a'}}\right\rceil \sum_{t\in\mathcal{T}_a}ty_{at}  &\geq  \sum_{i\in \mathcal{N}}\sum_{o \in \mathcal{O}_i}\left\lceil \frac{\psi^o_i}{q_{a'}}\right\rceil\Bar{u}^o_i,
    \end{align}
    where $q_{a'}=\min_{\{a\in \mathcal{A}|\Bar{v}_{a}=1\}}\{q_{a}\}$. Identical metric inequalities may be generated for two different arcs; thus, only unique inequalities are collected.

    While constraints \eqref{eq:MIbase} are always valid and therefore \rev{leave} the LP(BIN) relaxation \rev{unchanged}, one can \rev{obtain stronger} aggregations by generating metric inequalities using a point $(\Bar{x},\Bar{y})$ from a strengthened version of LP(BIN) (e.g., LP(BIN) with constraints~\eqref{eq:01valid}).
    We will use this approach in \rev{our Integrated Cut-Generation strategy outlined in \ref{ICG}.}

\subsection{Lagrangian Selection}
\rev{Generated} metric inequalities depend on the particular solution of LP(BIN) used to construct them. In this subsection, we introduce a method for selecting additional LP solutions from which further metric inequalities can be generated. We adapt an approach from \cite{fischetti2011relax,becu2024approximating}, originally developed for choosing effective Gomory mixed-integer cuts. The key idea is as follows: rather than repeatedly adding cuts which can introduce numerical issues~\cite{cornuejols2013safety,dey2018theoretical}, we incorporate them into the objective function and re-solve the LP relaxation. This modification yields a new LP-feasible point, which can then serve as the basis for generating additional metric inequalities in an iterative fashion.

\begin{algorithm}[!h]
\small
\SetAlgoLined
\DontPrintSemicolon
\SetKwInOut{Input}{Input}
\SetKwInOut{Output}{Output}

\Input{Initial LP solution $(\Bar{x},\Bar{y})$, commodity demands $\psi$}
\Output{Helper constraint set $\mathcal{M}^H$, Lagrangian-modified LP}

Initialize sets $\mathcal{M}^H \leftarrow \emptyset, \mathcal{M}^I \leftarrow \emptyset$\;
Calculate effective capacities: $\Bar{q}_a \leftarrow q_a\sum_{t\in\mathcal{T}_a}t\Bar{y}_{at}$ for all $a\in\mathcal{A}$\;

\For{each active arc $a \in \mathcal{A}$ (where $\Bar{q}_a > 0$)}{
    $(\bar{v}, \bar{u}) \leftarrow$ optimal solution to Separation Problem \eqref{odMI} given $\Bar{q}$ and $v_a=1$\;
    $z^* \leftarrow$ objective value of solution $(\bar{v}, \bar{u})$\;
    
    \If{$z^* = 0$}{
        Construct metric inequality $C_{metric}$ from $(\bar{v}, \bar{u})$ and $\psi$ using \eqref{eq:MIbase}\;
        \If{$C_{metric} \notin \mathcal{M}^H$}{
            $\mathcal{M}^H \leftarrow \mathcal{M}^H \cup \{ C_{metric} \}$\;
            $\mathcal{M}^I \leftarrow \mathcal{M}^I \cup \{ \text{Integer version of } C_{metric} \text{ via } \eqref{eq:intMI} \}$\;
        }
    }
}

Update LP: Add constraints $\mathcal{M}^I$\;
$(x^*,y^*) \leftarrow$ Solve updated LP\;
$\mathcal{D}^I \leftarrow$ Dual values corresponding to $\mathcal{M}^I$ in $(x^*,y^*)$\;
Update LP objective: Move $\mathcal{M}^I$ to objective using penalties $\mathcal{D}^I$\;

\Return $\mathcal{M}^H$, Modified LP, $(x^*,y^*)$

\caption{\rev{Lagrangian relaxation for metric inequality generation}} \label{alg:Lagrange}
\end{algorithm}

When collecting helper metric inequalities \eqref{eq:MIbase} to add directly to BIN, we use a Lagrangian relaxation approach to guide the generation of additional inequalities. In particular, the integral metric inequalities \eqref{eq:intMI} are incorporated into the objective of LP(BIN) with their associated Lagrange multipliers, and the modified LP is re-solved to obtain a new feasible solution. This updated solution can then be used to generate further metric inequalities. Algorithm~\ref{alg:Lagrange} summarizes the procedure for iteratively constructing and selecting both helper and integral metric inequalities.

\rev{
\subsection{Integrated Cut-Generation Strategy}\label{ICG}

In this section, we present an algorithmic framework that uses the new valid inequalities and also metric inequalities generated via Lagrangian relaxation. We call this framework the Integrated Cut-Generation (ICG) approach.
The goal of ICG is to construct an enhanced collection of valid inequalities
that improves the dual bounds obtained while keeping constraint-generation overhead manageable.}

\rev{
\paragraph{Strategy and implementation}
The implementation proceeds in two phases:}

\begin{enumerate}
    \item \rev{\textbf{Pre-solve strengthening.}} Prior to invoking the MIP solver, we generate helper metric inequalities as follows:
    \begin{enumerate}
        \item Apply Algorithm~\ref{alg:Lagrange} to LP(BIN) to generate and collect an initial set of helper metric inequalities; denote the resulting LP as LP(BIN)$_{\text{MI}}$.
        \item Using the optimal solution of LP(BIN)$_{\text{MI}}$, generate SAC-Pack constraints~\eqref{eq:01valid} via the separation IP~\eqref{sep} and Gen-SAC-Pack constraints~\eqref{eq:intValid} obtained through post-processing. Add these inequalities to form LP(BIN)$_{\text{MI+VI}}$.
        \item Apply Algorithm~\ref{alg:Lagrange} again to LP(BIN)$_{\text{MI+VI}}$ to obtain an additional set of helper metric inequalities.
        \item Add all helper metric inequalities collected in steps (a) and (c) to the BIN formulation.
    \end{enumerate}

    \item \rev{\textbf{Root-node cut generation.}} At the root node, we add the following user cuts via a callback:
    \begin{enumerate}
        \item SAC-Pack constraints~\eqref{eq:01valid};\label{sacpackpp}
        \item Post-processed Gen-SAC-Pack constraints~\eqref{eq:intValid} derived from (i);
        \item Gen-SAC-Pack constraints~\eqref{eq:intValid} with $B \le 3$ generated using the separation routine of Section~\ref{sec:intVI} (with dynamic programming details in 
        \ref{app:dp}).
    \end{enumerate}
\end{enumerate}

\section{Computational Study} \label{sec:comp}
We conduct computational experiments on two sets of instances to evaluate the effectiveness of our proposed methods on both path- and arc-based MCND models.
We generate fulfillment instances using the demand data and network topology of an e-commerce company to evaluate path-based problems
and use the publicly-available Canad instances \cite{CanadInst} to evaluate arc-based problems.
The objectives of our computational study are to (i) assess the value of reformulating INT using the multiple-choice binary capacity variables \eqref{eqn:replace}, (ii) measure the improvement to LP(BIN) (and the arc-based version) when applying the SAC-Pack constraints \eqref{eq:01valid} and the Gen-SAC-Pack constraints \eqref{eq:intValid}, and (iii) \rev{evaluate the effectiveness of the Integrated Cut-Generation (ICG) approach in improving the lower bounds obtained by commercial solvers.}

The optimization models and separation routines are coded in Python 3.10 and solved using Gurobi 11.0 with default settings for the IP solver, unless noted otherwise. 
We use a Linux computing
cluster, consisting of nodes using 24-core dual Intel Xeon Gold 6226 CPUs @ 2.7 GHz with 192GB of RAM each, to run all experiments. Reported times are in hours, unless otherwise noted.

\subsection{Instance Generation}
We use two sets of instances for this study: e-commerce fulfillment instances and the Canad instances \cite{CanadInst}. We use the e-commerce fulfillment instances to assess all previously described path-based methodologies and results, as well as \rev{the ICG approach.}
The Canad instances are only used to evaluate
the effectiveness of the SAC-Pack and Gen-SAC-Pack constraints in strengthening the LP relaxations for \textit{arc-based} models.
Thus, all experiments are conducted on the e-commerce fulfillment instances, unless otherwise noted.

\paragraph{E-commerce fulfillment instances} We generate synthetic instances constructed from historical order fulfillment and middle-mile network data provided by an e-commerce company. We create three instance groups, as defined by the number of vendors  
$\mathcal{N}_v\subset \mathcal{N}$, fulfillment centers (FCs) 
$\mathcal{N}_f\subset \mathcal{N}$, and destinations $\mathcal{N}_d\subset \mathcal{N}$. In each instance, the set of vendors $\mathcal{N}_v$ represents externally-owned locations from which the company fulfills orders. The set of FCs $\mathcal{N}_f$ represents the internal locations within the fulfillment network where the company stores products and fulfills orders. The set of destinations $\mathcal{N}_d$ represents last-mile distribution center locations that complete the delivery process to the final customer. Fulfillment centers act as both commodity origins and transfer centers; they can be used for intermediate shipment transfers along a path, whereas locations in $\mathcal{N}_v$ and $\mathcal{N}_d$ cannot. Thus, the total number of commodity origins for a given instance is $|\mathcal{N}_v|+|\mathcal{N}_f|$ and total number of arcs is at most $(|\mathcal{N}_v| + |\mathcal{N}_f|)|\mathcal{N}_d| + |\mathcal{N}_f|(|\mathcal{N}_v| + |\mathcal{N}_f|-1)$. 

For each group, we generate 5 randomized instances. Locations $\mathcal{N}_v,\ \mathcal{N}_f, \text{ and } \mathcal{N}_d$ are sampled from larger sets of vendor, FC, and destination locations, respectively, for each instance.  A set of potential commodities is initialized with all vendor-destination pairs. We then randomly remove a subset of commodities to better reflect the sparsity of real-world fulfillment networks, resulting in an average of 60\% of vendor-destination pairs in the commodity set. A set of commodities originating at FC locations is also created. For each destination $d\in\mathcal{N}_d$, we generate 5 commodities by assuming that the 5 nearest FCs send orders to $d$.
For context, this construction is designed to mimic the practice of using multiple FCs to stock the same items and then to sometimes require shipment of items from further FCs due to inventory stock availability. Thus, the number of FC-originating commodities is always $5|\mathcal{N}_d|$. Commodity demands are sampled from empirical demand distributions derived from company shipment volume data. 

We generate a set of arcs $\mathcal{A}$, consisting of both direct arcs $\mathcal{A}_d$ and consolidation arcs $\mathcal{A}_c$. Consolidation arcs connect all vendors to FCs, FCs to other FCs, and FCs to destinations, whereas direct arcs connect a vendor to a destination if a commodity for that pair exists. To generate a path set $\mathcal{P}_k$ for commodity $k\in \mathcal{K}$, we start by including the direct path (equivalent to an arc) connecting $o(k)$ to $d(k)$, with a length denoted as $l_a$. We then add the following four paths: (1) the shortest distance two-arc path using a single transfer FC, (2) the two-arc path transferring at the FC closest to the origin, (3) the two-arc path transferring at the FC closest to the destination, and (4) a three-arc path transferring at the FCs in (2) and (3), if they are not the same. We then include all network paths with a total length $\sum_{a'\in p}l_{a'}$ of $1.1l_a$ or less and containing a maximum of three arcs. Duplicate paths in each set $\mathcal{P}_k$ are removed.

Consolidation arcs transport commodity flow using a truckload freight mode with a single capacity $q_a=12000$ and cost dependent on the arc length $l_a$, whereas direct arcs can ship commodities using both truckload and less-than-truckload (LTL) freight modes. We use the same capacity and cost structure as \cite{greening2023lead} for direct arcs, as shown in Table \ref{tab:capCost}, where we use $l_a$ to denote the length of arc $a\in\mathcal{A}_d$ as measured in miles and $v_a$ to denote the volume flowing on arc $a \in\mathcal{A}_d$. We pre-process all direct arc costs and incorporate them into the variable path cost $c_p$. This pre-processing step enables us to reduce the problem size by eliminating direct arcs from the set of arcs in the MCND models.
\begin{table}[h]
\centering
\footnotesize
\caption{Freight mode costs per shipment on direct arc $a\in\mathcal{A}_d$.}
\begin{tabular}{lcc}
\hline
Freight Mode & Volume (lbs) & Cost \\ \hline
Truckload & $\phantom{2,00}0<v_a \leq 12,000$ & $750+1.27l_a$ \\
LTL$_1$ & $\phantom{2,00}0< v_a \leq \phantom{1}2,000$ & $0.05(750+1.27l_a)+v_a(0.234+0.0004l_a)$ \\
LTL$_2$ & $2,000<v_a\leq \phantom{1}2,700$ & $0.05(750+1.27l_a)+2000(0.234+0.0004l_a)$ \\
LTL$_3$ & $2,700 < v_a \leq \phantom{1}4,000$ & $0.8v_a(0.234+0.0004l_a)$ \\ \hline
\end{tabular}
\label{tab:capCost}
\end{table}

In Table \ref{tab:instChar}, we provide the following instance characteristics: group number, number of vendors, FCs, and destinations, instance number, number of commodities, total volume across all commodities, and number of arcs and paths. 

\begin{table}[h]
\footnotesize
\centering
\caption{E-commerce, path-based instance characteristics.}
\begin{tabular}{cccccrrrrlr}
\hline
\multirow{2}{*}{Group} & \multirow{2}{*}{Vend} & \multirow{2}{*}{FC} & \multirow{2}{*}{Dest} & \multirow{2}{*}{Inst} & \multicolumn{1}{c}{\multirow{2}{*}{Comm}} & \multicolumn{1}{c}{\multirow{2}{*}{\begin{tabular}[c]{@{}c@{}}Total Vol \\ (lbs)\end{tabular}}} & \multicolumn{2}{c}{Arcs} &  & \multicolumn{1}{c}{Paths} \\ \cline{8-9} 
 &  &  &  &  & \multicolumn{1}{c}{} & \multicolumn{1}{c}{} & \multicolumn{1}{c}{$|\mathcal{A}|$} & \multicolumn{1}{c}{$|\mathcal{A}_c|$} &  & \multicolumn{1}{c}{$|\mathcal{P}|$}\\ \hline
\multirow{5}{*}{1} & \multirow{5}{*}{20} & \multirow{5}{*}{5} & \multirow{5}{*}{15} & 1 & 252 & 641,443 & 365 & 188 &  & 972  \\
 &  &  &  & 2 & 264 & 782,451 & 382 & 193 &  & 996  \\
 &  &  &  & 3 & 263 & 741,524 & 380 & 192 &  & 989  \\
 &  &  &  & 4 & 262 & 747,582 & 380 & 193 &  & 1,093  \\
 &  &  &  & 5 & 254 & 606,122 & 368 & 189 &  & 1,012\vspace{2mm} \\
\multirow{5}{*}{2} & \multirow{5}{*}{90} & \multirow{5}{*}{9} & \multirow{5}{*}{55} & 1 & 3,040 & 2,177,671 & 4,113 & 1,348 &  & 18,639  \\
 &  &  &  & 2 & 3,577 & 2,461,797 & 4,677 & 1,375 &  & 20,006  \\
 &  &  &  & 3 & 3,071 & 2,167,907 & 4,166 & 1,370 &  & 20,082  \\
 &  &  &  & 4 & 2,841 & 1,995,813 & 3,930 & 1,364 &  & 17,570  \\
 &  &  &  & 5 & 2,958 & 2,075,331 & 4,048 & 1,365 &  & 16,931\vspace{2mm} \\
\multirow{5}{*}{3} & \multirow{5}{*}{105} & \multirow{5}{*}{10} & \multirow{5}{*}{65} & 1 & 4,233 & 2,313,167 & 5,698 & 1,790 &  & 29,727  \\
 &  &  &  & 2 & 3,812 & 2,133,729 & 5,264 & 1,777 &  & 28,295  \\
 &  &  &  & 3 & 4,402 & 2,432,839 & 5,856 & 1,779 &  & 28,966 \\
 &  &  &  & 4 & 4,375 & 2,427,937 & 5,837 & 1,787 &  & 30,009  \\
 &  &  &  & 5 & 4,325 & 2,475,692 & 5,787 & 1,787 &  & 30,441 \\ \hline
\end{tabular}
\label{tab:instChar}
\end{table}

\paragraph{Canad instances} To assess the strength of our new valid inequalities when using an arc-based formulation, we use a set of multicommodity network flow Canad instances \cite{CanadInst} used in several papers \cite{HewittMike2010CEaH,crainic2000simplex,crainic2001bundle,ghamlouche2003cycle,chouman2015cutting,chouman2017commodity,AKHAVANKAZEMZADEH2022255}. In the previously mentioned papers, the instances are used to benchmark solution approaches for the MCND problem with splittable demand; however, they remain feasible with unsplittable demands.
Thus, following \cite{hewitt2013branch}, we use the set identified as ``C" to study the unsplittable MCND problem. This set consists of 31 instances with the naming scheme \#nodes-\#arcs-\#commodities-cost attribute-capacity attribute. There are two cost attributes denoted as F for a high fixed-to-variable cost ratio and V otherwise. There are two capacity attributes denoted as T for tightly-capacitated and L for loosely-capacitated arcs relative to total demand. The instances range in size from 20-30 nodes, 230-700 arcs, and 40-400 commodities.

Unlike the e-commerce fulfillment instances, all nodes can act as transfer nodes, and a single node can be an origin for one commodity and a destination for another; thus, there is no notion of node sets like $\mathcal{N}_v$, $\mathcal{N}_f$, or $\mathcal{N}_d$. Furthermore, the Canad instances were developed for the case where arc-capacity binary variables indicate whether an arc, with its predefined capacity $q_a$, is turned on, as opposed to installing integer multiples of $q_a$. Thus, for our arc-based experiments, we use the same formulation as the authors of \cite{hewitt2013branch} to allow for a direct comparison (see 
\ref{app:arcBasedExp} for the arc-based, fixed-capacity formulation).

Lastly, while the e-commerce fulfillment Group 1 instances are comparable in size to the Canad instances, Groups 2 and 3 are significantly larger than those commonly used to evaluate MCND methodologies. We chose  these larger instances to better evaluate the potential of embedding our developed methodology into more sophisticated solution approaches for addressing real-world e-commerce fulfillment network problems.

\subsection{Comparison of Formulations}\label{sec:compForm}
We begin our computational experiments by demonstrating how reformulating INT \eqref{rblpSM} with binary multiple-choice capacity variables \eqref{eqn:replace}
can improve the lower bound solution found by a commercial solver. To accomplish this, we optimize both models, INT \eqref{rblpSM} and BIN \eqref{rblpB1}, for 3 hours using the best bound focus setting. In Table \ref{tab:binVsInt}, we report the number of capacity variables, capacity constraints, IP lower bounds, IP gaps to the best objective found, and the improvements to the IP gaps. All values are the average across the 5 instances of each group. The best objective used to calculate the IP gap for each instance is obtained by optimizing both models for 12 hours with default settings and selecting the minimum objective value of the two. We measure the improvement (Impr) by calculating the percent decrease of the IP gaps.

\begin{table}[htbp]
\footnotesize
\centering
\caption{Comparing INT \eqref{rblpSM} and BIN \eqref{rblpB1} MCND formulations.}
\begin{tabular}{ccccccc}
\hline
Group & Model Type & Cap Vars  & Cap Constrs & IP LB & IP Gap & Impr \\ \hline
\multirow{2}{*}{1} & INT & \phantom{1,}191.0 & \phantom{1,}191.0 & 189,577 & 0.9\% & - \\
 & BIN & \phantom{1,}394.8 & \phantom{1,}382.0 & 190,268 & 0.6\% & 38.5\%\vspace{2mm} \\
\multirow{2}{*}{2} & INT & 1,364.4 & 1,364.4 & 658,504 & 6.2\% & - \\
 & BIN & 2,182.8 & 2,728.8 & 663,859 & 5.4\% & 12.4\%\vspace{2mm} \\
\multirow{2}{*}{3} & INT & 1,784.0 & 1,784.0 & 747,035 & 8.3\% & - \\
 & BIN & 2,788.8 & 3,568.0 & 748,815 & 8.1\% & 2.6\%\\ \hline
\end{tabular}
\label{tab:binVsInt}
\end{table}

We first observe that even with an increase in model size ($\sum_{a\in\mathcal{A}_c}(T^{\max}_a-1)$ more variables and $|\mathcal{A}|$ more constraints), the commercial solver finds stronger lower bounds, on average, for BIN compared to INT. When analyzing the detailed results, we observe that BIN consistently outperforms in terms of best bound across \textit{every} instance in each group. 

We next observe that Group 1 shows the largest \rev{percent} improvement\rev{, whereas the absolute IP-gap reduction is largest for Group 2.} \rev{The stronger relative improvement in Group 1 reflects the fact that it contains a} higher proportion of arcs with $T^{\max}_a>1$, and thus a larger \rev{relative} increase in number of capacity variables. \rev{We believe} these results indicate that the BIN model is particularly effective in networks where many arcs may require more than one truckload.

Lastly, among the 5 instances in Group 1, a state-of-the-art commercial solver is only able to achieve optimality within 3 hours for the first instance, confirming the difficulty of these problems.

\subsection{Improvements to LP Relaxation} \label{subsec:LPImpr}
In this section, we evaluate the strength of the valid inequalities introduced in Section~\ref{sec:relax}, comparing them to existing methods. 

\paragraph{E-commerce Fulfillment Instances}
We begin by assessing the performance of the SAC-Pack and Gen-SAC-Pack constraints in strengthening the LP relaxation (LPR) solutions for e-commerce fulfillment instances using Group 1. 
Specifically, we compare the objective values
of LP(BIN) strengthened by disaggregated linking constraints \eqref{eqn:binDA}, 
\rev{$r$}-split $c$-strong inequalities with \rev{$r$} set to a maximum value of $10$ (applied using the same method described in the computational study of \citep{atamturk2002splittable}), and our new valid inequalities \eqref{eq:01valid} and \eqref{eq:intValid}. In each case, we continue generating additional cuts until no further improvement of the LP relaxation is observed. In \rev{Figure~\ref{fig:viPerf}}, we present the results for the configurations defined in Table~\ref{tab:viDefns}. The figure reports, for Group~1, the \rev{average} improvement in IP gap\rev{---with ranges indicating the smallest and largest improvements observed across the individual instances---}over LP(BIN) without valid inequalities (a)  and the \rev{average} number of cuts added. 
We \rev{order} the configurations \rev{by their improvement in} IP gap, 
where the IP gap is defined as:
\begin{equation}
    \frac{(\text{best IP objective}-\text{LPR})}{\text{best IP objective}}\label{eq:IPGap}.
\end{equation}
The improvement of a strengthened LP relaxation (LPR) over the base LP relaxation (LP(BIN))---equivalently, the improvement in IP gap---\rev{is computed as}:
\begin{equation}
    \frac{(\text{LPR}-\text{LP}(\text{BIN}))}{(\text{best IP objective}-\text{LP}(\text{BIN}))}\label{eq:lpImpr}.
\end{equation}

\begin{table}[htpb]
\footnotesize
\centering
\caption{Definitions of the valid inequality configurations reported. Configurations (b) through (\rev{g}) add the listed inequalities to (a).}
\begin{tabular}{cl}
\hline
Conf & \multicolumn{1}{l}{Definition} \\ \hline
(a) & LP(BIN) \\
(b) & Disaggregated linking constraints \eqref{eqn:binDA} \\
(c) & \rev{$r$}-split $c$-strong inequalities for $\rev{r}\leq 10$ \\
(d) & SAC-Pack constraints \eqref{eq:01valid} \\
(e) & (d) plus post-processing \\
(f) & (e) plus Gen-SAC-Pack constraints \eqref{eq:intValid} with $B\leq 3$ \\
(g) & (e) plus Gen-SAC-Pack constraints \eqref{eq:intValid} with $B\leq 10$ \\ \hline
\end{tabular}
\label{tab:viDefns}
\end{table}

In Figure~\ref{fig:viPerf}, we first observe that adding the disaggregated linking constraints (b) significantly strengthens the LP relaxation relative to LP(BIN)\rev{, achieving an improvement of approximately 65\%}.
We next observe that the \rev{$r$}-split $c$-strong inequalities (c) produce stronger relaxations than (b), as expected. \rev{However, on their own they} are significantly outperformed by the SAC-Pack constraints (d), which also add fewer cuts on average. 
As we continue, we see that each additional \rev{cut family, or the use of a larger upper bound $B$, further strengthens the LP relaxation. Ultimately, we obtain an average improvement of 79.8\%, reducing the IP gap from 8.7\% to 5.0\%.

We also note that configuration (g) requires the longest time to generate all cuts (187 seconds on average), while configuration (f) is the next longest at 93 seconds on average.}
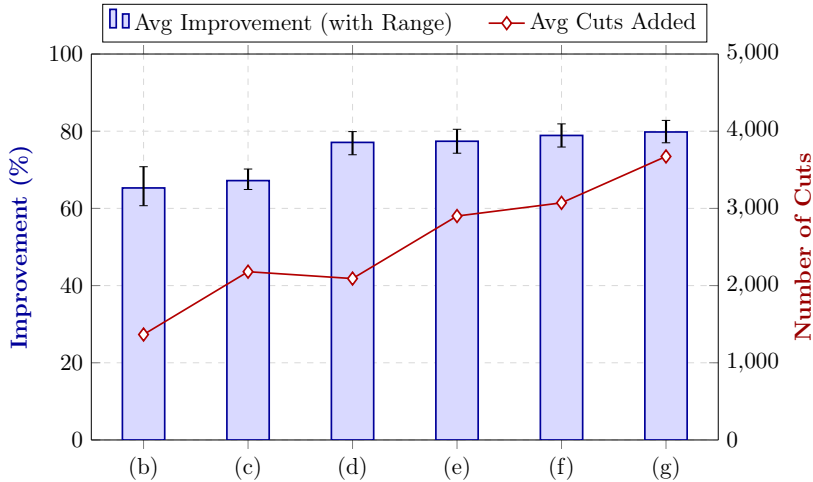
\begin{figure}[h!]
\centering
\begin{tikzpicture}[scale=0.8]
    \begin{axis}[
        width=12cm, height=8cm,
        axis y line*=left,    
        ybar,                 
        bar width=20pt,
        ylabel={Improvement (\%)},
        symbolic x coords={(b),(c),(d),(e),(f),(g)},
        xtick=data,
        ymin=0, ymax=100,
        grid=major,
        grid style={dashed, gray!30},
        ylabel style={font=\bfseries, color=blue!60!black},
        legend style={draw=none} 
    ]
    
    \addplot+[
        color=blue!60!black,
        fill=blue!15,
        thick,
        error bars/.cd,
        y dir=both, 
        y explicit,    
        error bar style={line width=1pt, black, sharp plot}
    ] coordinates {
        ((b), 65.3) -= (0,4.6)  += (0,5.5)  
        ((c), 67.2) -= (0,2.3)  += (0,3.0)  
        ((d), 77.1) -= (0,3.2)  += (0,2.8)  
        ((e), 77.4) -= (0,3.1)  += (0,3.1)  
        ((f), 78.9) -= (0,3.0)  += (0,3.0)  
        ((g), 79.8) -= (0,2.8)  += (0,3.0)  
    };
    
    \end{axis}
    
    \begin{axis}[
        width=12cm, height=8cm,
        axis y line*=right,   
        axis x line=none,     
        ylabel={Number of Cuts},
        ylabel style={font=\bfseries, color=red!60!black},
        ymin=0, ymax=5000,
        symbolic x coords={(b),(c),(d),(e),(f),(g)},
        xtick=data,
        legend style={
            at={(0.5,1.03)}, 
            anchor=south, 
            legend columns=-1,
            draw=black,       
            fill=white,       
            /tikz/every even column/.append style={column sep=0.5cm}
        }
    ]
    
    \addlegendimage{ybar, ybar legend, fill=blue!15, draw=blue!60!black, thick}
    \addlegendentry{Avg Improvement (with Range)}
    
    \addplot+[
        sharp plot,
        color=red!70!black,
        mark=diamond*,
        mark size=3pt,
        thick,
        mark options={fill=white}
    ] coordinates {
        ((b), 1366)
        ((c), 2180)
        ((d), 2091)
        ((e), 2901)
        ((f), 3072)
        ((g), 3673)
    };
    \addlegendentry{Avg Cuts Added}
    
    \end{axis}
\end{tikzpicture}
\caption{Valid inequalities (defined in Table \ref{tab:viDefns}) applied to the Group 1 instances.}\label{fig:viPerf}
\end{figure}

\paragraph{Canad Instances}
As previously noted, the SAC-Pack and Gen-SAC-Pack valid inequalities can also be generated for arc-based formulations.
To assess their effectiveness \rev{in this setting},
we compare the strengthened LP relaxation solutions for the arc-based Canad instances \rev{using a subset of the cut configurations listed in} Table \ref{tab:viDefns}.
In Table \ref{tab:lpArcSum}, we provide a summary of results for the Canad instances. Specifically, for each configuration we test, we provide the average values 
for the
LP relaxation solution (LPR), number of cuts added, cut generation time in seconds, IP gap, and the improvement of the IP gap. We also include the number of instances where the optimal objective is achieved after adding the cuts (\# Opt).

We use Gurobi with default settings and a time limit of 2 hours
to obtain the best IP objective for each instance.
Of the 31 instances, Gurobi finds the optimal objective for 18 instances.
Although some instances remain unsolved within 2 hours, the average IP gap is only 0.64\% \rev{(1.52\% among the unsolved instances)} with an average objective of $194,443$. 
For the LP relaxation solutions and IP gaps for individual instances, see 
\ref{app:arcBasedRes}.

The results in Table \ref{tab:lpArcSum} confirm that our new valid inequalities also significantly strengthen arc-based formulations. Our first observation is that adding only the disaggregated capacity linking constraints (b) improves the base LP by 79\%, but this improvement comes at the cost of adding, on average, more than 90,000 additional constraints. We next observe that the \rev{$r$}-split $c$-strong inequalities (c) achieve an 83\% improvement with only a fraction of the cuts compared to (b). However, the SAC-Pack constraints (d) yield an even greater improvement of over 85\%, reducing the average IP gap of (c) by \rev{0.5 percentage points, or by 15\% in relative terms}.
Notably, although the total number of SAC-Pack constraints generated is roughly twice that of the \rev{$r$}-split $c$-strong inequalities, the SAC-Pack constraints require less time to generate on average.

\begin{table}[htpb]
\footnotesize
\centering
\caption{\rev{Summarized results for the Canad instances.}}
\begin{tabular}{ccccccc}
\hline
Conf & \multicolumn{1}{c}{LPR} & \multicolumn{1}{c}{Cuts} & \multicolumn{1}{c}{\begin{tabular}[c]{@{}c@{}}Cut\\ Time (s)\end{tabular}} & \multicolumn{1}{c}{\begin{tabular}[c]{@{}c@{}}IP \\ Gap\end{tabular}} & \multicolumn{1}{c}{Impr} & \multicolumn{1}{c}{\# Opt} \\ \hline
(a) & 163,198 & \phantom{10,00}0 & \phantom{10}0 & 18.9\% & \phantom{0}0.0\% & 0 \\
(b) & 185,881 & 93,245 & \phantom{0}18 & \phantom{0}3.9\% & 79.3\% & 0 \\
(c) & 188,246 & \phantom{1}1,697 & 168 & \phantom{0}3.3\% & 82.7\% & 2 \\
(d) & 191,160 & \phantom{1}3,306 & 115 & \phantom{0}2.8\% & 85.3\% & 2 \\
(f) & 191,229 & \phantom{1}7,371 & 746 & \phantom{0}2.7\% & 85.7\% & 2 \\
(d),(c) & 191,591 & \phantom{1}3,959 & 245 & \phantom{0}2.4\% & 87.2\% & 3 \\ \hline
\end{tabular}
\label{tab:lpArcSum}
\end{table}

We next observe that adding Gen-SAC-Pack constraints to SAC-Pack constraints (denoted as (f)) 
does not appear necessary for these instances and problem definition (i.e., without integer multiples of capacities on arcs). This conclusion is based on the marginal improvement obtained relative to the additional cuts and, more critically, the increased generation time.
We next test the configuration in which all SAC-Pack constraints are generated first, followed by the \rev{$r$}-split $c$-strong inequalities (denoted as (d),(c)). This configuration
yields a notable improvement in IP gap \rev{(87.2\% versus 85.7\% for (f))} while significantly reducing cut-generation time.
\rev{Overall, the SAC-Pack constraints (d) offer a good balance between cut strength, cut volume, and generation time. Adding the $r$-split $c$-strong inequalities (c) provides additional improvement while requiring only a modest increase in cuts and cut-generation time.}

We additionally compare configurations (d) and (d),(c) to the final bound reported in \cite{hewitt2013branch} (their Table 8). That bound is obtained by dynamically adding disaggregated capacity linking constraints and cover inequalities to their extended LP formulation at the root node (via user cuts) and then executing their branch-and-price procedure for 30 minutes. When comparing the bounds achieved by configurations (d) and (d),(c), we find they outperform the final bounds in \cite{hewitt2013branch} for 20 and 29 of the 31 instances, respectively (see Table \ref{tab:lpArc1} 
for detailed instance results).
While this is a somewhat unfair comparison (given advances in modern commercial solvers, which aid the separation routines of SAC-Pack inequalities relative to those available when \cite{hewitt2013branch} was written), we report this result simply to \rev{show}
the potential improvements our new valid inequalities could offer within such branch-and-price procedures.

\subsection{Evaluation of \rev{Integrated Cut-Generation Strategy}} \label{subsec:MIPImpr}

\rev{We next examine the performance of the Integrated Cut-Generation (ICG) approach. Using a 3-hour time limit for both methods, we compare ICG-strengthened BIN (dentoted 3h-ICG) with the baseline BIN (3h-BIN) formulation to assess improvements in dual bounds and overall solution quality.}

In Table \ref{tab:smMIP}, we compare the results of 3h-BIN and 3h-ICG for Group 1. Impressively, we observe that with 3h-ICG, we can now solve 4 out of 5 instances, with an average IP gap of 0.1\%, whereas with 3h-BIN, we can only solve the first instance, resulting in an average IP gap of 0.6\%. Furthermore, the total runtime (sum of constraint generation and solve time) is an hour or less for each of the four instances solved using 3h-ICG, with a substantial improvement (a 94\% decrease) for instance 1 compared to 3h-BIN. Although we are still unable to solve instance 4 within 3 hours, we decrease the IP gap by 74\% to 0.4\%.  On average, the 3h-ICG models include 175 helper metric inequalities, and the solver uses approximately 30 user cuts during its bounding procedure.

\begin{table}[htbp]
\footnotesize
\centering
\caption{Group 1 results for 3h-BIN and 3h-ICG.}
\begin{tabular}{cccccccc}
\hline
\multirow{2}{*}{Inst} & \multicolumn{2}{c}{3h-BIN} &  & \multicolumn{2}{c}{3h-ICG} &   \multirow{2}{*}{\begin{tabular}[c]{@{}c@{}}IP Gap \\ Impr\end{tabular}} & \multirow{2}{*}{\begin{tabular}[c]{@{}c@{}}Time \\ Impr\end{tabular}} \\ \cline{2-3} \cline{5-6}
 & IP Gap & Runtime &  & IP Gap & Runtime &    &  \\ \hline
1 & 0.0\% & 2.1 &  & 0.0\% & 0.1 &   \phantom{10}0\% & 94\% \\
2 & 0.4\% & 3.0 &  & 0.0\% & 0.5 &   100\% & 82\% \\
3 & 0.5\% & 3.0 &  & 0.0\% & 0.6 &   100\% & 79\% \\
4 & 1.6\% & 3.0 &  & 0.4\% & 3.0 &   \phantom{1}74\% & \phantom{1}0\% \\
5 & 0.4\% & 3.0 &  & 0.0\% & 1.0 &   100\% & 67\% \\ \hline
\end{tabular}
\label{tab:smMIP}
\end{table}

For Groups 2 and 3, we report the results for 3h-BIN, 12h-BIN, and 3h-ICG in Table \ref{tab:g56mip}. For 3h-ICG, the difference between 3 hours and the solve time is the time it takes to generate all inequalities prior to optimization. The most significant observation is that 3h-ICG (with limited solve time) produces stronger bounds than those found when solving BIN for 12 hours. On average, \rev{this corresponds to} IP gap improvements of 26.5\% and 22.5\% for Groups 2 and 3, respectively. 
Thus, we find that even for larger instances, using the ICG strategy yields significantly stronger lower bounds than those obtained by solving BIN for extended time limits.

\begin{table}[htbp]
\footnotesize
\centering
\caption{Groups 2 and 3 results for 12h-Base and 3h-ICG.}
\begin{tabular}{cccccccc}
\hline
\multirow{2}{*}{Group} & \multirow{2}{*}{Inst} & \multirow{2}{*}{\begin{tabular}[c]{@{}c@{}}3h-BIN \\ IP Gap\end{tabular}} & \multirow{2}{*}{\begin{tabular}[c]{@{}c@{}}12h-BIN\\  IP Gap\end{tabular}} & \multicolumn{2}{c}{3hr-ICG} & \multirow{2}{*}{\begin{tabular}[c]{@{}c@{}}3h-BIN\\ Impr\end{tabular}} & \multirow{2}{*}{\begin{tabular}[c]{@{}c@{}}12h-BIN\\  Impr\end{tabular}} \\ \cline{5-6}
 &  &  &  & Solve Time & IP Gap &  &  \\ \hline
\multirow{5}{*}{2} & 1 & 5.2\% & 5.1\% & 2.5 & 3.9\% & 25.3\% & 24.3\% \\
 & 2 & 5.2\% & 5.1\% & 2.3 & 4.2\% & 19.5\% & 18.2\% \\
 & 3 & 5.6\% & 5.5\% & 2.2 & 3.6\% & 35.3\% & 33.8\% \\
 & 4 & 5.9\% & 5.7\% & 2.4 & 3.9\% & 33.4\% & 32.1\% \\
 & 5 & 5.3\% & 5.1\% & 2.5 & 3.9\% & 25.9\% & 24.2\%\vspace{2mm} \\
\multirow{5}{*}{3} & 1 & 8.1\% & 7.1\% & 1.5 & 5.5\% & 32.9\% & 22.7\% \\
 & 2 & 7.2\% & 6.8\% & 1.5 & 4.9\% & 32.3\% & 28.8\% \\
 & 3 & 7.7\% & 6.7\% & 1.5 & 5.1\% & 33.0\% & 23.2\% \\
 & 4 & 9.7\% & 8.7\% & 1.6 & 7.1\% & 26.6\% & 18.4\% \\
 & 5 & 7.5\% & 6.8\% & 1.0 & 5.5\% & 27.1\% & 19.2\% \\ \hline
\end{tabular}
\label{tab:g56mip}
\end{table}\vspace{-3mm}

\section{Conclusions}
In this paper, we presented a reformulation of the multicommodity capacitated network design (MCND) problem which redefines the integer capacity variables as a multiple-choice selection of binary variables. We studied a structured relaxation of this formulation and showed that the convex hull solutions 
can be described using inequalities of a certain form.
Using this result, we defined two new classes of valid inequalities for the MCND problem with unsplittable flow: the single-arc commodity packing (SAC-Pack) and generalized single-arc commodity packing (Gen-SAC-Pack) constraints. We also developed separation routines for each class.
We next described metric inequalities and a Lagrangian relaxation approach to generate them,  \rev{and then developed the Integrated Cut-Generation (ICG) approach, which incorporates both the new valid inequalities and the metric inequalities.}
Finally, we presented computational results conducted on the path-based e-commerce logistics instances,
as well as on the arc-based Canad instances. We found that solvers obtain better lower bounds for the reformulation using multiple-choice binary capacity variables (BIN) than for
the integer capacity formulation (INT). Our computational study also showed the strength of the newly defined SAC-Pack and Gen-SAC-Pack inequalities in comparison to existing methods for both path- and arc-based models. 
Our findings ultimately indicate that using a strategy \rev{such as the ICG approach} can substantially strengthen the bounds obtained by commercial solvers.
Although we have defined and tested one strategy which yields stronger bounds than the base model, its components can also be included into alternative bounding procedures, depending on the application.

\appendix
\section{Disaggregated Capacity Constraints Example}\label{app:intBinDA}
Consider the following example from where the solution of LP(BIN) 
with constraints \eqref{eqn:binDA} produces a stronger solution compared to LP(INT)
with constraints \eqref{eqn:binDA}.

Assume we have 2 commodities $k=1$ and $k=2$ that can both traverse arc $a$ using paths $p=1$ and $p=2$, respectively.
Let demand $d_1 = 5$, demand $d_2=105$, and $q_a=100$. Thus, $t^{\min}_{a1} = 1$, whereas $t^{\min}_{a2} = 2$. Let the LP relaxation solution be $x_1=0.75$ and $x_2 = 0.5$. The aggregrated linking constraints \eqref{8thSM} (and, similarly, \eqref{8thB1} using \eqref{eqn:replace}) state:
$$5x_1+105x_2\leq 100\tau_a.$$
Thus, the installed capacity on arc $a$ is $0.5625$ units for LP(INT) and LP(BIN).

To strengthen LP(INT),
we can add the following constraints \eqref{eqn:intDA}:
\begin{align*}
    x_1&\leq \tau_a,\\
    2x_2&\leq \tau_a.
\end{align*}
Because $x_2=0.5$, the installed capacity increases to $1$ unit (or $\tau_a = 1$).

To strengthen LP(BIN), 
we can add the following constraints \eqref{eqn:binDA}:
\begin{align*}
    x_1&\leq y_{a1}+y_{a2},\\
    x_2&\leq y_{a2}.
\end{align*}
This would set $y_1=0.25$ and $y_2=0.5$ (because $y_1$ has a lower penalty cost), which is a total installed capacity of $1.25$ units.

Thus, linking constraints \eqref{eqn:binDA} added to LP(BIN) can produce a stronger relaxation solution compared to LP(INT) with linking constraints \eqref{eqn:intDA} if there exists a subset of commodities with demands that exceed an arc's single-unit capacity.


\section{Separation Using Dynamic Programming.}\label{app:dp}

We would like to improve the runtime of solving (\ref{eq:SEPchild}).
\begin{observation}[Decomposition] Let 
\begin{eqnarray*}
S^t := \left\{ x \,|\, \sum_{p\in\mathcal{P}_a}w_p x_p \leq q_at \right\} \ \forall \, t \in \mathcal{T}_a,
\end{eqnarray*}
then we can solve:
\begin{eqnarray}\label{eq:DPbasic}
\begin{array}{rcl}
\textup{FEASVAL}^t:= &\textup{max}_{x,y}& \theta^{* \top} x - \alpha^{* }_t\\
&\textup{s.t.}& (x,y) \in S^t.
\end{array}
\end{eqnarray}
Note that:
\begin{itemize}
\item $\textup{FEASVAL}= \textup{max}\{\textup{max}_t \{ \textup{FEASVAL}^t\}, 0\}$ 
\item If $\textup{FEASVAL} \neq 0$, then $\hat{t} \in \textup{argmax}_t\{ \textup{FEASVAL}^t\}$ and $x^{\hat{t}} \in \textup{argmax}_x S^{\hat{t}}$, then $(x^{\hat{t}}, e_{\hat{t}})$ is an optimal solution of (\ref{eq:SEPchild}). If $\textup{FEASVAL} = 0$, then $(0,0)$ is an optimal solution of (\ref{eq:SEPchild}).
\end{itemize}
\end{observation}
\subsection{Improving DP running time.}
Observe now that (\ref{eq:DPbasic}) is a knapsack problem. Moreover, we know that $\|\theta\|_{\infty} \leq B$, as we artificially constrained this in (\ref{eq:SEPmaster}).   We can exploit this feature in the following fashion.

Let us consider a general knapsack of the following form:
\begin{eqnarray}\label{eq:KS}
\begin{array}{rcl}
\textup{OPT}:= &\textup{max}& \sum_{j = 1}^n \theta_j x_j \\
&\textup{s.t.}&\sum_{j = 1}^n w_j x_j \leq q_a, \\
&& x \in \{0, 1\}^n.
\end{array}
\end{eqnarray}

We ``guess'' an optimal value of (\ref{eq:KS}) is $\gamma$. Consider the following min-knapsack:
\begin{eqnarray}\label{eq:minKS}
\begin{array}{rcl}
\textup{VAL}(\gamma):= &\textup{min}& \sum_{j = 1}^n w_j x_j \\
&\textup{s.t.}&\sum_{j = 1}^n \theta_j x_j \geq \gamma, \\
&& x \in \{0, 1\}^n.
\end{array}
\end{eqnarray}

\paragraph{DP running-time} The key fact is that since $\theta$'s are small, any guessed value of $\gamma $ is not too large; in particular $\gamma \leq n\cdot \|\theta\|_{\infty}$. So the DP to solve (\ref{eq:minKS}) takes $\mathcal{O}(n\cdot n\|\theta\|_{\infty}) = \mathcal{O}(n^2B)$ time, which could be significantly smaller time than solving (\ref{eq:KS}) which takes $\mathcal{O}(n\cdot q_a)$ time since $q_a$ can be super large and $B$ is quite small.

To solve (\ref{eq:KS}) using (\ref{eq:minKS}), we begin with an observation.

\begin{observation}\label{obs:rel}
Observe that:
\begin{itemize}
\item If we guessed too low, i.e, $\gamma \leq \textup{OPT}$, then $\textup{VAL}(\gamma) \leq q_a$.
\item If we guessed too high, i.e, $\gamma >  \textup{OPT}$, then $\textup{VAL}(\gamma) > q_a$.
\end{itemize}
\end{observation}
We can use Observation (\ref{obs:rel}) to design a bisection search as follows:
\begin{enumerate}
\item (Pre-process) Let $\textup{LB} = 0$ and $\textup{UB} = \sum_{j = 1}^n \theta_j$. Solve (\ref{eq:KS}) with $\gamma = \textup{UB} $. If $\textup{VAL}(\textup{UB}) \leq q_a$, STOP, as we have solved the problem and the optimal solution of (\ref{eq:minKS}) is an optimal solution of (\ref{eq:KS}). 
Else:
\item Let $\gamma = \left\lceil\frac{1}{2}\cdot (\textup{LB} + \textup{UB})\right\rceil$. Solve (\ref{eq:KS}).
\item If $\textup{VAL}(\gamma) > q_a$, then set $\textup{UB} \leftarrow \gamma$. If $\textup{VAL}(\gamma) \leq q_a$, then set $\textup{LB} \leftarrow \gamma$. If $\textup{LB} = \textup{UB}$, STOP (and the optimal solution of (\ref{eq:minKS}) is an optimal solution of (\ref{eq:KS})), else GOTO STEP 2.
\end{enumerate}
\paragraph{Overall running time} The number of iterations of bisection search is known to be $\mathcal{O}\left(\textup{log}(\sum_{j = 1}^n \theta_j)\right)$. Combining everything (accounting for the running-time of solving (\ref{eq:minKS})), we get a running-time of $\mathcal{O}(n^2B\cdot\textup{log}(nB))$ which can be much better than $\mathcal{O}(n\cdot q_a)$.

\section{Solver Performances with Helper Metric Inequalities}\label{app:solvers}
To assess whether the benefits of helper metric inequalities are specific to Gurobi, we also utilize CPLEX and SCIP to solve \eqref{rblpB1} (3h-BIN)  and \eqref{rblpB1} plus two rounds of helper metric inequalities added (3h-MI2; as described in Section \ref{sec:helper} and Algorithm \ref{alg:Lagrange}) for Group 3 of the e-commerce fulfillment instances. We provide a runtime of 3 hours and solve the models using default settings for all solvers.
In Table \ref{tab:solverComp}, we report the solvers and versions tested, the resulting average IP gaps calculated using the best found objective, and the average improvement to the IP gap for each solver after adding the helper metric inequalities to BIN \eqref{rblpB1}. Also note that while the runtime for each setup was 3 hours, the time to generate the helper metric inequalities took an average of 1.5 hours, meaning the 3h-MI2 models had an average solve time limit of 1.5 hours.

\begin{table}[htpb]
\footnotesize
\centering
\caption{Comparison of solver performance using Group 3 instances.}
\begin{tabular}{ccccc}
\hline
\multirow{2}{*}{Solver} & \multirow{2}{*}{Version} & \multicolumn{2}{c}{IP Gap} & \multirow{2}{*}{\begin{tabular}[c]{@{}c@{}}IP Gap \\ Impr\end{tabular}} \\ \cline{3-4}
 &  & 3h-BIN & 3h-MI2 &  \\ \hline
Gurobi & 11.0.0 & 8.3\% & 6.0\% & 27.1\% \\
CPLEX & 22.1.1.0 & 9.3\% & 5.9\% & 36.4\% \\
SCIP & 9.0.0 & 8.7\% & 7.6\% & 12.9\% \\ \hline
\end{tabular}
\label{tab:solverComp}
\end{table}

Consistent with the results in Section \ref{subsec:MIPImpr}, the addition of helper metric inequalities results in all three solvers finding stronger lower bounds through improved cutting-plane generation, with CPLEX seemingly benefiting the most.

\section{Arc Flow Conservation Constraints}\label{app:arcBased}
Let $x^o_{ij}\in \{0,1\}$ equal 1 if commodity $o\in\mathcal{O}$ is transported via arc $(i,j)\in\mathcal{A}$, and 0 otherwise. 
The arc flow conservation constraints are formulated as follows:
\begin{align}
    \sum_{(i,j)\in \delta^+(i)}x^o_{ij} - \sum_{(j,i)\in \delta^-(i)}x^o_{ji} &= \begin{cases}
        \psi^o_o, \quad i=o,\\
        0, \quad i \neq o,d,\\
        \psi_d^o, \quad i = d,
    \end{cases}
    & \forall \ i\in \mathcal{N}^o,\ \forall \ o\in \mathcal{O},\label{eq:dual1}
\end{align}
where $\delta^+(i)$ and $\delta^-(i)$ refer to the sets of arcs emanating from and ending at node $i$, respectively, and set $\mathcal{N}^o$ represents the nodes contained in path set $\mathcal{P}_o$ for commodity $o\in\mathcal{O}$.

\section{Arc-Based Formulation When Solving CANAD Instances}\label{app:arcBasedExp}
Similar to the problem definition in Section \ref{sec:form}, let $c_{ij}$ represent the variable cost of transporting one unit of demand on arc $(i,j)\in\mathcal{A}$, $f_{ij}$ represent the fixed cost of activating (i.e., installing one unit of capacity) on arc $(i,j)\in\mathcal{A}$, let $q_{ij}$ represent the capacity of arc $(i,j)\in\mathcal{A}$, and let $d_k$ represent the demand of commodity $k\in\mathcal{K}$ with origin $o(k)$ and destination $d(k)$.
Let $x^k_{ij}\in \{0,1\}$ be a binary variable indicating if commodity $k\in\mathcal{K}$ is transported via arc $(i,j)\in\mathcal{A}$ or not. Let $y_{ij}\in \{0,1\}$ be a binary variable indicating if arc $(i,j)\in\mathcal{A}$ is activated to transport commodity demands or not. The arc-based unsplittable MCND problem is:
{\small \begin{mini!}[2] 
	{\scriptstyle x,y}{\sum_{k\in\mathcal{K}}\sum_{(i,j) \in \mathcal{A}}c_{ij}d_kx^k_{ij} +
	\sum_{(i,j)\in \mathcal{A}}f_{ij} y_{ij}}{\label{A1}}{\label{ObjA1}}
	\addConstraint{\sum_{(i,j)\in \delta^+(i)}x^k_{ij} - \sum_{(j,i)\in \delta^-(i)}x^k_{ji}}{=\begin{cases}
        \phantom{-}1, \quad i=o(k),\\
        \phantom{-}0, \quad i \neq o(k),d(k),\\
        -1, \quad i = d(k),
    \end{cases}}{\  \forall \,  i \in \mathcal{N},\ \forall \, k \in \mathcal{K} \label{1stA1}}
	\addConstraint{\sum_{k\in \mathcal{K}}d_kx^k_{ij}}{\leq q_{ij}y_{ij},\quad}{  \ \forall \,  (i,j)\in \mathcal{A} \label{8thA1}}
	\addConstraint{x^k_{ij}}{\in \{0,1\},}{\ \forall \, (i,j) \in \mathcal{A}}, \  \forall \,  k \in \mathcal{K} \label{xvarA1}
	\addConstraint{y_{ij}}{\in \{0,1\},}{\ \forall \, (i,j) \in \mathcal{A},\label{fvarA1}} 
 \end{mini!} }
where $\delta^+(i)$ and $\delta^-(i)$ refer to the sets of arcs emanating from and ending at node $i\in\mathcal{N}$, respectively. Constraints \eqref{1stA1} ensure flow balance. Constraints \eqref{8thA1} activate arc $(i,j)\in\mathcal{A}$ and ensure its capacity is not exceeded.

When noted, we use the following disaggregated capacity linking constraints strengthen the LP relaxation:
$$x^k_{ij}\leq y_{ij}, \quad \ \forall \, k\in \mathcal{K},\ \forall \, (i,j)\in\mathcal{A}.$$

\section{Detailed CANAD Instance Results}\label{app:arcBasedRes}
In Table \ref{tab:lpArc1}, we provide the individual CANAD instance results that were summarized in Table \ref{tab:lpArcSum}. Results in bold indicate the optimal objective was achieved for that configuration. 
\begin{table}[htbp]
\tiny
\centering
\caption{\rev{Arc-based results for the CANAD instances.} Objectives for (d) and (d),(c) are marked with an asterisk if the LPR objective was greater than the final bound found in \cite{hewitt2013branch}.}
\Rotatebox{90}{%
\begin{tabular}{lrrrlrrlrrlrrlrrlrr}
\hline
\multirow{2}{*}{Instance} & \multicolumn{1}{c}{\multirow{2}{*}{Obj}} & \multicolumn{2}{c}{(a)} & & \multicolumn{2}{c}{(b)} & & \multicolumn{2}{c}{(c)} & & \multicolumn{2}{c}{(d)} & & \multicolumn{2}{c}{(f)} & & \multicolumn{2}{c}{(d),(c)} \\ \cline{3-4} \cline{6-7} \cline{9-10} \cline{12-13} \cline{15-16} \cline{18-19} 
 & \multicolumn{1}{c}{} & \multicolumn{1}{c}{LPR} & \multicolumn{1}{c}{IP Gap} & & \multicolumn{1}{c}{LPR} & \multicolumn{1}{c}{IP Gap} & & \multicolumn{1}{c}{LPR} & \multicolumn{1}{c}{IP Gap} & & \multicolumn{1}{c}{LPR} & \multicolumn{1}{c}{IP Gap} & & \multicolumn{1}{c}{LPR} & \multicolumn{1}{c}{IP Gap} & & \multicolumn{1}{c}{LPR} & \multicolumn{1}{c}{IP Gap} \\ \hline
20-230-40-V-L & 423,933 & 378,623 & 10.69\% &  & 422,853 & 0.25\% &  & \textbf{423,933} & 0.00\% &  & \textbf{423,933}* & 0.00\% &  & \textbf{423,933} & 0.00\% &  & \textbf{423,933}* & 0.00\% \\
20-230-40-V-T & 398,870 & 359,175 & 9.95\% &  & 368,819 & 7.53\% &  & 378,656 & 5.07\% &  & 397,618* & 0.31\% &  & 397,616 & 0.31\% &  & 397,826* & 0.26\% \\
20-230-40-F-T & 668,699 & 584,306 & 12.62\% &  & 633,466 & 5.27\% &  & 642,157 & 3.97\% &  & 667,972* & 0.11\% &  & 667,972 & 0.11\% &  & \textbf{668,699}* & 0.00\% \\
20-230-200-V-L & 94,644 & 68,805 & 27.30\% &  & 91,301 & 3.53\% &  & 91,302 & 3.53\% &  & 90,902\phantom{*} & 3.95\% &  & 90,954 & 3.90\% &  & 91,312\phantom{*} & 3.52\% \\
20-230-200-F-L & 138,084 & 98,343 & 28.78\% &  & 132,036 & 4.38\% &  & 132,036 & 4.38\% &  & 130,842\phantom{*} & 5.24\% &  & 130,741 & 5.32\% &  & 132,048* & 4.37\% \\
20-230-200-V-T & 98,610 & 73,976 & 24.98\% &  & 95,669 & 2.98\% &  & 95,743 & 2.91\% &  & 94,661* & 4.01\% &  & 94,790 & 3.87\% &  & 95,762* & 2.89\% \\
20-230-200-F-T & 137,639 & 100,728 & 26.82\% &  & 131,544 & 4.43\% &  & 131,556 & 4.42\% &  & 130,884\phantom{*} & 4.91\% &  & 130,865 & 4.92\% &  & 131,596* & 4.39\% \\
20-300-40-V-L & 430,253 & 386,462 & 10.18\% &  & 427,947 & 0.54\% &  & \textbf{430,253} & 0.00\% &  & \textbf{430,253}* & 0.00\% &  & \textbf{430,253} & 0.00\% &  & \textbf{430,253}* & 0.00\% \\
20-300-40-F-L & 597,059 & 492,690 & 17.48\% &  & 575,255 & 3.65\% &  & 584,975 & 2.02\% &  & 589,451* & 1.27\% &  & 589,451 & 1.27\% &  & 589,537* & 1.26\% \\
20-300-40-V-T & 501,766 & 448,170 & 10.68\% &  & 460,930 & 8.14\% &  & 479,890 & 4.36\% &  & 500,047* & 0.34\% &  & 500,047 & 0.34\% &  & 500,301* & 0.29\% \\
20-300-40-F-T & 643,395 & 546,666 & 15.03\% &  & 596,839 & 7.24\% &  & 614,690 & 4.46\% &  & 641,608* & 0.28\% &  & 641,608 & 0.28\% &  & 642,043* & 0.21\% \\
20-300-200-V-L & 75,777 & 58,934 & 22.23\% &  & 73,127 & 3.50\% &  & 73,153 & 3.46\% &  & 72,837\phantom{*} & 3.88\% &  & 73,104 & 3.53\% &  & 73,219* & 3.38\% \\
20-300-200-F-L & 117,722 & 87,878 & 25.35\% &  & 110,926 & 5.77\% &  & 110,961 & 5.74\% &  & 109,940\phantom{*} & 6.61\% &  & 109,953 & 6.60\% &  & 110,998* & 5.71\% \\
20-300-200-V-T & 76,281 & 61,482 & 19.40\% &  & 74,003 & 2.99\% &  & 74,097 & 2.86\% &  & 74,013* & 2.97\% &  & 74,242 & 2.67\% &  & 74,143* & 2.80\% \\
20-300-200-F-T & 109,330 & 84,733 & 22.50\% &  & 103,633 & 5.21\% &  & 103,682 & 5.17\% &  & 103,241* & 5.57\% &  & 103,353 & 5.47\% &  & 103,818* & 5.04\% \\
30-520-100-V-L & 54,387 & 44,118 & 18.88\% &  & 53,023 & 2.51\% &  & 53,246 & 2.10\% &  & 53,322* & 1.96\% &  & 53,325 & 1.95\% &  & 53,428* & 1.76\% \\
30-520-100-F-L & 95,877 & 67,560 & 29.53\% &  & 90,174 & 5.95\% &  & 90,542 & 5.56\% &  & 90,477* & 5.63\% &  & 90,651 & 5.45\% &  & 90,839* & 5.25\% \\
30-520-100-V-T & 53,812 & 46,616 & 13.37\% &  & 51,326 & 4.62\% &  & 51,883 & 3.58\% &  & 53,069* & 1.38\% &  & 53,089 & 1.34\% &  & 53,136* & 1.26\% \\
30-520-100-F-T & 99,195 & 76,710 & 22.67\% &  & 94,011 & 5.23\% &  & 94,854 & 4.38\% &  & 95,407* & 3.82\% &  & 95,435 & 3.79\% &  & 95,844* & 3.38\% \\
30-520-400-V-L & 114,317 & 96,691 & 15.42\% &  & 111,763 & 2.23\% &  & 111,813 & 2.19\% &  & 111,540\phantom{*} & 2.43\% &  & 111,493 & 2.47\% &  & 111,835* & 2.17\% \\
30-520-400-F-L & 150,729 & 122,680 & 18.61\% &  & 146,680 & 2.69\% &  & 146,778 & 2.62\% &  & 146,121* & 3.06\% &  & 146,209 & 3.00\% &  & 146,898* & 2.54\% \\
30-520-400-V-T & 116,255 & 100,649 & 13.42\% &  & 114,061 & 1.89\% &  & 114,232 & 1.74\% &  & 113,908\phantom{*} & 2.02\% &  & 114,122 & 1.84\% &  & 114,301\phantom{*} & 1.68\% \\
30-520-400-F-T & 154,842 & 126,357 & 18.40\% &  & 149,751 & 3.29\% &  & 149,939 & 3.17\% &  & 149,506* & 3.45\% &  & 149,713 & 3.31\% &  & 150,012* & 3.12\% \\
30-700-100-V-L & 47,883 & 39,055 & 18.44\% &  & 47,308 & 1.20\% &  & 47,497 & 0.81\% &  & 47,445* & 0.91\% &  & 47,445 & 0.91\% &  & 47,544* & 0.71\% \\
30-700-100-F-L & 60,384 & 45,709 & 24.30\% &  & 58,207 & 3.60\% &  & 58,372 & 3.33\% &  & 58,367\phantom{*} & 3.34\% &  & 58,390 & 3.30\% &  & 58,613* & 2.93\% \\
30-700-100-V-T & 47,670 & 40,992 & 14.01\% &  & 45,078 & 5.44\% &  & 45,822 & 3.88\% &  & 46,645* & 2.15\% &  & 46,719 & 1.99\% &  & 46,899* & 1.62\% \\
30-700-100-F-T & 56,686 & 44,810 & 20.95\% &  & 53,661 & 5.34\% &  & 54,485 & 3.88\% &  & 55,085* & 2.83\% &  & 55,240 & 2.55\% &  & 55,183* & 2.65\% \\
30-700-400-V-L & 98,613 & 79,059 & 19.83\% &  & 96,605 & 2.04\% &  & 96,612 & 2.03\% &  & 96,023\phantom{*} & 2.63\% &  & 96,212 & 2.43\% &  & 96,623* & 2.02\% \\
30-700-400-F-L & 136,549 & 106,564 & 21.96\% &  & 130,724 & 4.27\% &  & 130,758 & 4.24\% &  & 130,143\phantom{*} & 4.69\% &  & 130,050 & 4.76\% &  & 130,822* & 4.19\% \\
30-700-400-V-T & 96,545 & 81,217 & 15.88\% &  & 94,012 & 2.62\% &  & 94,104 & 2.53\% &  & 93,696\phantom{*} & 2.95\% &  & 93,976 & 2.66\% &  & 94,183* & 2.45\% \\
30-700-400-F-T & 131,923 & 109,364 & 17.10\% &  & 127,572 & 3.30\% &  & 127,611 & 3.27\% &  & 127,016* & 3.72\% &  & 127,150 & 3.62\% &  & 127,668* & 3.23\% \\ \hline
Average & 194,443 & 163,198 & 18.9\% &  & 185,881 & 3.9\% &  & 188,246 & 3.3\% &  & 191,160\phantom{*} & 2.8\% &  & 191,229 & 2.7\% &  & 191,591\phantom{*} & 2.4\% \\ \hline
\end{tabular}
}
\label{tab:lpArc1}
\end{table}

\clearpage
 
 \bibliographystyle{apalike}
 \bibliography{References}

\end{document}